\input ./style/arxiv-vmsta.cfg
\documentclass[numbers,compress,v1.0.1]{vmsta}
\usepackage{vtexbibtags}
\usepackage{mathrsfs}

\volume{7}
\issue{1}
\pubyear{2020}
\firstpage{17}
\lastpage{41}
\aid{VMSTA149}
\doi{10.15559/20-VMSTA149}
\articletype{research-article}

\startlocaldefs
\newcommand{\lleft}{\left}
\newcommand{\rright}{\right}
\newtheorem{theorem}{Theorem}
\newtheorem{proposition}{Proposition}
\newtheorem{assumption}{Assumption}

\theoremstyle{definition}
\newtheorem{remark}{Remark}
\newtheorem{example}{Example}
\newtheorem{definition}{Definition}

\newcommand{\rrVert}{\Vert}
\newcommand{\llVert}{\Vert}

\allowdisplaybreaks
\ifdefined\HCode 
\def\index#1{}
\else 
\fi
\endlocaldefs

\begin{document}

\begin{frontmatter}
\pretitle{Research Article}

\title{Pathwise asymptotics for Volterra processes conditioned to a noisy version of the Brownian motion}
\author{\inits{B.}\fnms{Barbara}~\snm{Pacchiarotti}\ead[label=e1]{pacchiar@mat.uniroma2.it}}
\address{\institution{Dept. of Mathematics, University of Rome ``Tor Vergata''},
via~della~Ricerca~Scientifica, 00133, Roma, \cny{Italy}}

\markboth{B. Pacchiarotti}{Pathwise asymptotics for Volterra processes}





\begin{abstract}
In this paper we investigate a problem of large deviations for continuous
Volterra processes under the influence of model disturbances. More
precisely, we study the behavior, in the near future after $T$, of a
Volterra process driven by a Brownian motion in a case where the Brownian
motion is not directly observable, but only a noisy version is observed or
some linear functionals of the noisy version are observed. Some examples
are discussed in both cases.
\end{abstract}

\received{\sday{30} \smonth{7} \syear{2019}}
\revised{\sday{30} \smonth{11} \syear{2019}}
\accepted{\sday{10} \smonth{2} \syear{2020}}
\publishedonline{\sday{27} \smonth{2} \syear{2020}}
\end{frontmatter}

\section{Introduction}\label{sec1}

In this paper we study the asymptotics of the regular conditional prediction
law of a Gaussian Volterra process\index{Gaussian Volterra process} in a case where one does not observe the
process directly, but instead observes a noisy version of it. More precisely
we consider two different situations which generalize the results contained
in \cite{Pac} and \cite{Gio-Pac}, respectively. Let $X=(X_{t})_{t\geq 0}$ be a
continuous real Volterra process.\index{continuous real Volterra process}

\begin{definition}\label{def:volterra representation}
A centered Gaussian process\index{Gaussian process} $X$ is a Volterra process\index{Volterra process} if, for every $T>0$, it
admits the representation
%
\begin{equation}
\label{eqn:integral-representation} X_{t} = \int_{0}^{T}
K(t,s) \, dB_{s},
\end{equation}
where $B=( B_{t})_{t\geq 0}$ is a Brownian motion\index{Brownian motion} and $K$ is a square
integrable function on $[0,T]^{2}$ (the kernel) such that $K(t,s)=0$ for all
$s > t$.
\end{definition}

For a Volterra process\index{Volterra process} the covariance function is
%
\begin{eqnarray}
\label{eqn:Volterra-covariance} k(t,s)= \int_{0}^{s\wedge t}
K(t,u)K(s,u)\,du \quad \mbox{for }t,s \in [0,T].
\end{eqnarray}
Let $\tilde{B}=(\tilde{B}_{t})_{t\geq 0}$ be another Brownian motion\index{Brownian motion}
independent of $B$ and for $\alpha , \tilde{\alpha }\in {\mathbb{R}}$ define
$W^{\alpha ,\tilde{\alpha }}= \alpha B + \tilde{\alpha }\tilde{B}$.

\textbf{First case} For fixed $n\in {\mathbb{N}}$ and $T>0$, we consider the
conditioning of $X$ on $n$ linear functionals of the paths of $W^{\alpha
,\tilde{\alpha }}$,
\begin{equation*}
\boldsymbol{ G}_{T}\bigl(W^{\alpha ,\tilde{\alpha }}\bigr)=
\bigl(G_{T}^{1}\bigl(W^{ \alpha
,\tilde{\alpha }}\bigr),\ldots
,G_{T}^{n}\bigl(W^{\alpha ,\tilde{\alpha }}\bigr)
\bigr)^{\intercal
},
\end{equation*}
more precisely,
\begin{equation*}
\boldsymbol{ G}_{T}\bigl(W^{\alpha ,\tilde{\alpha }}\bigr)= \int
_{0}^{T} \boldsymbol{ g}(t)\,dW^{\alpha ,\tilde{\alpha }}_{t}=
\Biggl(\int_{0}^{T} g_{1}(t)\,
dW^{\alpha ,\tilde{\alpha }}_{t},\ldots ,\int_{0}^{T}
g_{n}(t) \, dW^{\alpha
,\tilde{\alpha }}_{t} \Biggr)^{\intercal },
\end{equation*}
where $\boldsymbol{ g} =(g_{1},\ldots ,g_{n})^{\intercal }$ is a suitable
vectorial function defined on $[0,T]$. Informally the generalized conditioned
process $X^{\boldsymbol{ g;x}}$, for $\boldsymbol{x}\in {\mathbb{R}}^{n}$, is
the law of the Gaussian process\index{Gaussian process} $X$ conditioned on the set
\begin{equation*}
\Biggl\{ \int_{0}^{T} \boldsymbol{
g}(t)\,dW^{\alpha ,\tilde{\alpha }}_{t}= \boldsymbol{x} \Biggr\}=\bigcap
_{i=1}^{n} \Biggl\{ \int
_{0}^{T} g_{i}(t)\,dW^{
\alpha ,\tilde{\alpha }}_{t}=x_{i}
\Biggr\}.
\end{equation*}
We obtain a large deviation principle for the family of processes
$((X^{\boldsymbol{ g;x}}_{T+\varepsilon t}-\break X^{\boldsymbol{ g;x}}_{T})_{t \in
[0,1]})_{\varepsilon >0}$.

\textbf{Second case} We are interested in the regular conditional law of the
process $X$ given the $\sigma $-algebra ${\mathscr{F}}^{\alpha ,\tilde{\alpha
}}_{T}$, where $({\mathscr{F}}^{\alpha ,\tilde{\alpha }}_{t})_{t\geq 0}$ is
the filtration generated by the mixed Brownian motion\index{Brownian motion} $W^{\alpha
,\tilde{\alpha }}$, i.e. we want to condition the process to the past of the
mixed Brownian motion\index{Brownian motion} up to a fixed time $T>0$. Informally the generalized
conditioned process $X^{{ \psi }}$, for $\psi $ being a continuous function, is the
law of the Gaussian process\index{Gaussian process} $X$ conditioned on the set
\begin{equation*}
\bigl\{ W^{\alpha ,\tilde{\alpha }}_{t}=\psi _{t}, \, t\in [0,T]
\bigr\}.
\end{equation*}

Here we obtain a large deviation principle for the family of processes
$((X^{\psi }_{T+\varepsilon t}- X^{\psi }_{T})_{t\in [0,1]})_{ \varepsilon
>0}$.

Since $T, \alpha $ and $\tilde{\alpha }$ are fixed positive numbers the
dependence (in the notations) from these quantities will be omitted.

The paper is organized as follows. In Section~\ref{sect:ldp} we recall some
basic facts on large deviation theory for continuous Gaussian processes\index{Gaussian process} and
Volterra processes.\index{Volterra process} Sections~\ref{sect:cond1} and~\ref{sect:cond2} are
dedicated to the main results. Both are divided into three subsections. In
the first one we give the conditional law,\index{conditional law} in the second one we prove the
large deviation principle and in the third one we present some examples.
Section~\ref{sect:cond1} is dedicated to the conditioning on $n$ functionals
of the paths of the noisy process. Section~\ref{sect:cond2} is dedicated to
the conditioning on the past of the noisy process.

\section{Large deviations for continuous Gaussian processes\index{Gaussian process}}\label{sect:ldp}

We briefly recall some main facts on large deviations principles we are going
to use. For a detailed development of this very wide theory we can refer, for
example, to the following classical references: Chapter II in Azencott
\cite{Aze}, Section 3.4 in Deuschel and Strook \cite{Deu-Str}, Chapter 4 (in
particular Sections 4.1, 4.2 and 4.5) in Dembo and Zeitouni \cite{Dem-Zei}.

\begin{definition}\label{def:ldp}
Let $E$ be a topological space, ${\mathscr{B}}(E)$ be the Borel $\sigma
$-algebra and $(\mu _{\varepsilon })_{\varepsilon >0}$ be a family of
probability measures on ${\mathscr{B}}(E)$. We say that the family of
probability measures $(\mu _{\varepsilon })_{\varepsilon >0}$ satisfies a
large deviation principle on $E$ with the rate function $I$ and the inverse
speed $\eta _{\varepsilon }$ ($\eta _{\varepsilon }>0$, $\eta _{\varepsilon
}\rightarrow 0$ as $\varepsilon \to 0$) if, for any open set $\Theta $,
\begin{equation*}
-\inf_{x \in {\Theta } } I(x) \le \liminf_{\varepsilon \to 0}{\eta
_{\varepsilon }} \log \mu _{\varepsilon }(\Theta )
\end{equation*}
and for any closed set $\Gamma $,
\begin{equation*}
\limsup_{\varepsilon \to 0}{\eta _{\varepsilon }} \log \mu
_{\varepsilon }(\Gamma ) \le -\inf_{x \in {\Gamma }} I(x).
\end{equation*}
\end{definition}

A rate function is a lower semicontinuous mapping $I:E\rightarrow [0,+\infty
]$. A rate function $I$ is said to be  \textit{good} if the sets $\{I\le a\}$ are
compact for every $a \ge 0$.

In this paper $E$ will be the set of continuous functions on $[0,1]$ and
${\mathscr{B}}(E)$ will be the Borel $\sigma $-algebra generated by the open sets
induced by the uniform convergence. Therefore in this section we consider
process in the interval $[0{,}1]$. Let $U\,{=}\,(U_{t})_{t\in [0,1]}$, be a
continuous and centered Gaussian process on a probability space $(\Omega
,{\mathscr{F}}, {\mathbb{P}})$. From now on, we will denote by $C [0, 1]$ the
set of continuous functions on $[0, 1]$, and by  ${\mathscr{B}}(C [0, 1])$ the
Borel $\sigma $-algebra generated by the open sets induced by the uniform
convergence. Moreover, we will denote by ${\mathscr{M}}[0, 1]$ its dual, that
is, the set of signed Borel measures on $[0, 1]$. The action of
${\mathscr{M}}[0, 1]$ on $C [0, 1]$ is given by
\begin{equation*}
\langle \lambda ,h \rangle =\int_{0}^{1} h(t)
\,d\lambda (t), \quad \lambda \in {\mathscr{M}}[0, 1], \,h\in C [0, 1].
\end{equation*}

\begin{remark}%
We say that a family of continuous processes $((U_{t}^{\varepsilon })_{t \in
[0,1]})_{\varepsilon >0}$ satisfies a large deviation principle if the
associated family of laws satisfy a large deviation principle on $ C [0,1]$.
\end{remark}

The following remarkable theorem (Proposition 1.5 in \cite{Aze}) gives an
explicit expression of the Cram\'{e}r transform $\Lambda ^{*}$ of a
continuous centered Gaussian process $(U_{t})_{t \in [0,1]}$ with covariance
function $k$. Let us recall that for $\lambda \in {\mathscr{M}}[0,1]$,
\begin{equation*}
\Lambda (\lambda )=\log {\mathbb{E}}\bigl[\exp \bigl(\langle U, \lambda \rangle
\bigr)\bigr]= \frac{1}{2} \int_{0}^{1}
\int_{0}^{1} k(t,s) \, d\lambda (t) \,d \lambda
(s).
\end{equation*}

\begin{theorem}\label{th:cramer-transform}
Let $(U_{t})_{t \in [0,1]}$ be a continuous and centered Gaussian process
with covariance function $k$. Let $\Lambda ^{*}$ denote the Cram\'{e}r
transform of $\Lambda $, that is
\begin{eqnarray*}
\Lambda ^{*}(x) &=& \sup_{\lambda \in {\mathscr{M}}[0,1]} \bigl( \langle
\lambda , x \rangle - \Lambda (\lambda ) \bigr)
\\*
&=& \sup_{\lambda \in {\mathscr{M}}[0,1]} \Biggl( \langle \lambda , x \rangle -
\frac{1}{2} \int_{0}^{1} \int
_{0}^{1} k(t,s) \, d\lambda (t) \,d\lambda (s)
\Biggr).
\end{eqnarray*}
Then,
\begin{equation*}
\Lambda ^{*}(x) = %
\begin{cases}
\frac{1}{2} \|x \|_{\mathscr{H}}^{2}, & x \in {\mathscr{H}},
\\
+\infty, & \text{otherwise},
\end{cases} %
\end{equation*}
where ${\mathscr{H}}$ and $\| . \|_{\mathscr{H}}$ denote, respectively, the
reproducing kernel\index{reproducing kernel Hilbert space} Hilbert space and the related norm associated to the
covariance function $k$.
\end{theorem}

Reproducing kernel\index{reproducing kernel Hilbert space} Hilbert spaces are an important tool to handle Gaussian
processes.\index{Gaussian process} For a detailed development of this wide theory we can refer, for
example, to Chapter 4 in \cite{Hid-Hit} (in particular Section~4.3) and
Chapter 2 in \cite{Ber-Tho} (in particular Sections~2.2 and~2.3). In order to
state a large deviation principle for a family of Gaussian processes,\index{Gaussian process} we need
the following definition.

\begin{definition}
A family of continuous processes ${((U^{\varepsilon }_{t})_{t \in
[0,1]}})_{\varepsilon >0}$ is exponentially tight at the inverse speed $\eta
_{\varepsilon }$, if for every $R>0$ there exists a compact set $K_{R}$ such
that
\begin{equation*}
\limsup_{\varepsilon \to 0} \eta _{\varepsilon }\log {\mathbb{P}}
\bigl(U^{\varepsilon }\notin K_{R}\bigr) \le -R.
\end{equation*}
\end{definition}

If the means and the covariance functions of an exponentially tight family of
Gaussian processes\index{Gaussian process} have a good limit behavior, then the family satisfies a
large deviation principle, as stated in the following theorem which is a
consequence of the classic abstract G\"{a}rtner--Ellis Theorem (Baldi Theorem
4.5.20 and Corollary 4.6.14 in \cite{Dem-Zei}) and
Theorem~\ref{th:cramer-transform}.

\begin{theorem}\label{th:ldp-gaussian}
Let $((U_{t}^{\varepsilon })_{t \in [0,1]})_{\varepsilon >0}$ be an
exponentially tight family of continuous Gaussian processes\index{Gaussian process} at 
the inverse speed function $\eta _{\varepsilon }$. Suppose that, for any
$\lambda \in {\mathscr{M}}[0,1]$,
\begin{equation*}
\lim_{\varepsilon \to 0} {\mathbb{E}} \bigl[ \bigl\langle \lambda ,
U^{\varepsilon }\bigr\rangle \bigr] = 0
\end{equation*}
and the limit
\begin{equation*}
\Lambda (\lambda ) = \lim_{\varepsilon \to 0}
\frac{1}{\eta _{\varepsilon
}}{\operatorname{Var}} \bigl( \bigl\langle \lambda ,
U^{\varepsilon }\bigr\rangle \bigr) = \int_{0}^{1}
\int_{0}^{1} {k}(t,s) \, d\lambda (t) \,d
\lambda (s)
\end{equation*}%
exists for some continuous, symmetric, positive definite function ${k}$,
that is the covariance function of a continuous Gaussian process,\index{Gaussian process} then
$((U_{t}^{\varepsilon })_{t \in [0,1]})_{\varepsilon >0}$ satisfies a large
deviation principle on $C[0,1]$ with the inverse speed $\eta _{\varepsilon }$
and the good rate function
\begin{equation*}
I(h) = %
\begin{cases}
\frac{1}{2}  \llVert   h  \rrVert   ^{2}_{{{\mathscr{H}} }}, & h \in {{ \mathscr{H}} },\\
+\infty, & \text{otherwise},
\end{cases} %
\end{equation*}
where ${{{\mathscr{H}}}}$ and $ \llVert   .  \rrVert   _{{{\mathscr{H}} }}$,
respectively, denote the reproducing kernel\index{reproducing kernel Hilbert space} Hilbert space and the related norm
associated to the covariance function ${k}$.
\end{theorem}

In order to prove exponential tightness\index{exponential tightness} we shall use the following result
(see Proposition 2.1 in \cite{Mac-Pac}).

\begin{proposition}\label{prop:exp-tight}
Let $((U_{t}^{\varepsilon })_{t \in [0,1]})_{\varepsilon >0}$ be a family of
continuous Gaussian processes,\index{Gaussian process} where $U_{0}^{\varepsilon }=0$ for all
$\varepsilon >0$. Suppose there exist constants $\beta ,M_{1},M_{2}>0$ such
that for $\varepsilon >0$
\begin{equation*}
\sup_{s,t\in [0,1],s\neq t} \frac{|{\mathbb{E}}[U^{\varepsilon
}_{t}-U^{\varepsilon }_{s}]|}{|t-s|^{\beta }} \leq M_{1}
\end{equation*}%
and
%
\begin{equation}\label{eqn:cov-condition-exp-tight}
\sup_{s,t\in [0,1],s\neq t} \frac{\operatorname{Var}(U^{\varepsilon
}_{t}-U^{\varepsilon }_{s})}{\eta _{\varepsilon }\,|t-s|^{2\beta }} \leq
M_{2}%
\end{equation}
then $((U_{t}^{\varepsilon })_{t \in [0,1]})_{\varepsilon >0}$ is
exponentially tight at 
the inverse speed function~$\eta_{\varepsilon }$.
\end{proposition}

\begin{remark}\label{rem:exp-equiv}
Suppose $ ((U_{t}^{\varepsilon })_{t \in [0,1]})_{\varepsilon >0} $ is a
family of centered Gaussian processes, defined on the probability space
$(\Omega ,\mathscr{F},\mathbb{P})$, that satisfies a large deviation
principle on $C[0,1]$ with the inverse speed $\eta _{\varepsilon }$ and the
good rate function $I $. Let $(m^{\varepsilon })_{\varepsilon >0}\subset
C[0,1]$, $m\in C[0,1] $ be functions such that $ m^{\varepsilon
}\overset{C[0,1]}{\underset{}{\longrightarrow }} {m} $, as $ \varepsilon \to
0 $. Then, the family of processes $ (m^{\varepsilon }+U^{\varepsilon
})_{\varepsilon >0} $ satisfies a large deviation principle on $C[0,1] $ with
the same inverse speed $\eta _{\varepsilon }$ and the good rate function
\begin{equation*}
I_{m}(h) =I(h-m)= %
\begin{cases}
\frac{1}{2}  \llVert   h-m  \rrVert   ^{2}_{{{\mathscr{H}} }}, & h- m \in {{
\mathscr{H}} },\\
+\infty, & h- m \notin {{\mathscr{H}} }.
\end{cases} %
\end{equation*}
In fact the two families $(m^{\varepsilon }+U^{\varepsilon })_{\varepsilon
>0}$ and $(m +U^{\varepsilon })_{\varepsilon >0}$ are exponentially
equivalent (at the inverse speed $\eta _{\varepsilon }$) and therefore as far
as the large deviation principle is concerned, they are indistinguishable.
See Theorem 4.2.13 in \cite{Dem-Zei}.
\end{remark}

Our first aim is to study the behavior of the covariance function and of the
mean function of the original process $X$ in order to get a functional large
deviation principle for the family $((X_{T+\varepsilon t}-X_{T})_{t\in
[0,1]})_{\varepsilon >0}$, as $\varepsilon \to 0$.

Let $ (X_{t})_{t\geq 0} $ be a continuous centered Gaussian processes and fix
$T>0$. The next two assumptions guarantee that Theorem~\ref{th:ldp-gaussian}
is applicable to the family of processes $((X_{T+\varepsilon t}-X_{T})_{t\in
[0,1]})_{\varepsilon >0}$. Let $\gamma _{\varepsilon }>0$ be an infinitesimal
function, i.e. $\gamma _{\varepsilon }\to 0$ for $\varepsilon \to 0$.

\begin{assumption}\label{ass:limit-cov}
For any fixed $T>0$ there exists an asymptotic covariance function~$\bar{k}$
defined as
%
\begin{align} %
\bar{k}(t,s) &=\lim_{\varepsilon \to 0}
\frac{\operatorname{Cov}(X_{T+\varepsilon t}-X_{T}, X_{T+\varepsilon
s}-X_{T})}{\gamma ^{2}_{\varepsilon }}\nonumber
\\
&=\displaystyle \lim_{\varepsilon \to 0} \frac{k(T+\varepsilon t,
T+\varepsilon s) -k(T+\varepsilon t, T)-k( T+\varepsilon s,T) +k(T,
T)}{\gamma ^{2}_{\varepsilon }},  \label{eqn:k-bar}
\end{align}
uniformly in $(t,s)\in [0,1]\times [0,1]$.
\end{assumption}

\begin{remark}%
Notice that $\bar{k}$ is a continuous covariance function,\index{continuous covariance function} being the
(uniform) limit of continuous, symmetric and positive definite functions.
\end{remark}

\begin{remark}
Recall that the continuity of the covariance function is not a sufficient
condition to identify a Gaussian process\index{Gaussian process} with continuous paths. We need some
more regularity. Since we are investigating continuous Volterra processes,\index{Volterra process} it
would be useful to have a criterion to establish the regularity of the paths.
A sufficient condition for the continuity of the trajectories of a centered
Gaussian process can be given in terms of the metric entropy induced by the
canonical metric associated to the process (for further details, see
\cite{Dud1} and \cite{Dud2}). Such approach may be difficult to apply.
However, in \cite{Azm-Sot-Vii-Yaz}, a necessary and sufficient condition for
the H\"{o}lder continuity of a centered Gaussian process is established in
terms of the H\"{o}lder continuity of the covariance function. More precisely
a Gaussian process\index{Gaussian process} $(X_{t})_{t\in [0,T]}$ is H\"{o}lder continuous of
exponent $0<a<A$ if and only if for every $\varepsilon >0$, $s,t\in [0,T]$,
there exists a constant $c_{\varepsilon }>0$ such that
\begin{equation*}
{\mathbb{E}}\bigl[(X_{t}-X_{s})^{2}\bigr] \leq c_{\varepsilon }|t-s|^{A-
\varepsilon }.
\end{equation*}
Although, obviously, the H\"{o}lder continuity property of the process is
stronger than continuity, in many cases of interest this is more easily
established because the covariance function is not difficult to study.
Recalling the form of the covariance of a Volterra process\index{Volterra process}
(\ref{eqn:Volterra-covariance}) we have the following sufficient condition
for the H\"{o}lder continuity of a Volterra process:\index{Volterra process} there exist constants $ c,A
> 0 $ such that
%
\begin{equation}\label{eqn:mod-cont}
M(\delta )\leq c\,\delta ^{A}
\end{equation}
for all $ \delta \in [0,T] $, where
\begin{equation*}
M(\delta )= \sup_{\{t_{1},t_{2} \in [0,T]: |t_{1}-t_{2}|\leq \delta \}}
\int_{0}^{T} \bigl|K(t_{1},s)-K(t_{2},s)\bigr|^{2}\,ds.
\end{equation*}

From now on with covariance \textit{regular enough} we mean that the
covariance function satisfies some sufficient condition to ensure that the
associate process has continuous paths.
\end{remark}

\begin{assumption}\label{ass:exp-tight}
For any fixed $T>0$ there exist constants $M, \tau >0$, such that for
$\varepsilon >0$,
\begin{align*}
&\sup_{s,t\in [0,1],s\neq t}\frac{\operatorname{Var}(X_{T+\varepsilon
t}-X_{T+\varepsilon s})}{\gamma ^{2}_{\varepsilon }\,|t-s|^{2\tau }} \\
&=\sup _{s,t\in [0,1],s\neq t}\frac{k(T+\varepsilon t, T+\varepsilon t) -2
k(T+\varepsilon t, T+\varepsilon s ) +k(T+\varepsilon s , T+\varepsilon
s)}{\gamma ^{2}_{\varepsilon }\,|t-s|^{2\tau }} \leq M.
\end{align*}
\end{assumption}

As an immediate application of Theorem \ref{th:ldp-gaussian} (take
$U^{\varepsilon }_{t}=X_{T+\varepsilon t}-X_{T}$),
Assumptions~\ref{ass:limit-cov} and~\ref{ass:exp-tight} imply, if $\bar{k}$
is regular enough, that the family $((X_{T+\varepsilon t}-X_{T})_{t\in
[0,1]})_{\varepsilon >0}$ satisfies a large deviation principle on $C[0,1]$
with the inverse speed $\gamma _{\varepsilon }^{2}$ and the good rate
function given by
\begin{equation*}
J_{X}(h)=
\begin{cases}
\frac{1}{2} \|h\|^{2}_{\bar{{\mathscr{H}}}}, & h\in \bar{{\mathscr{H}}},
\\
+\infty, & \mbox{otherwise,}
\end{cases}
\end{equation*}
where $\bar{{\mathscr{H}}}$ is the reproducing kernel\index{reproducing kernel Hilbert space} Hilbert space
associated to the covariance function $\bar{k}$ and the symbol $\|\cdot
\|_{\bar{{\mathscr{H}}}}$ denotes the usual norm defined on
$\bar{{\mathscr{H}}}$.

In fact Assumption \ref{ass:limit-cov} immediately implies that
\begin{equation*}
\Lambda (\lambda )=\lim_{\varepsilon \to 0} \frac{\operatorname{Var}(\langle
\lambda , X_{T+\varepsilon \cdot }-X_{T}\rangle )}{\gamma ^{2}_{\varepsilon
}} =\int_{0}^{1}\int _{0}^{1} \bar{k}(t,s)\lambda (dt)\lambda (ds).
\end{equation*}
Furthermore, Assumption~\ref{ass:exp-tight} implies that the family
$((X_{T+\varepsilon t}-X_{T})_{t\in [0,1]})_{\varepsilon >0}$ is
exponentially tight 
at the inverse speed function $\gamma
^{2}_{\varepsilon }$.

\section{Conditioning to $n$ functionals of the path}\label{sect:cond1}

\subsection{Conditional law\index{conditional law}}\label{sec3.1}

Let $(\Omega , {\mathscr{F}}, ({\mathscr{F}}_{t})_{t\geq 0}, {\mathbb{P}})$
be a filtered probability space. On this space we consider a Brownian motion\index{Brownian motion}
$ B=( B)_{t\geq 0}$, a continuous real Volterra process $X=(X_{t})_{t\geq 0}$\index{continuous real Volterra process}
and another Brownian motion\index{Brownian motion} $\tilde{B}=(\tilde{B})_{t\geq 0}$ independent of
$B$. For $\alpha , \tilde{\alpha }\in {\mathbb{R}}$ let us define the mixed
Brownian motion\index{Brownian motion} $W^{\alpha ,\tilde{\alpha }}= \alpha B + \tilde{\alpha
}\tilde{B}$.

For fixed $n\in {\mathbb{N}}$ and $T>0$, we consider the conditioning of $X$
on $n$ linear functionals of $\boldsymbol{ G}_{T}(W^{\alpha ,\tilde{\alpha
}})= (G_{T}^{1}(W^{ \alpha ,\tilde{\alpha }}),\ldots ,G_{T}^{n}(W^{\alpha
,\tilde{\alpha }}))^{\intercal }$ of the paths of $W^{\alpha ,\tilde{\alpha
}}$,
\begin{equation*}
\boldsymbol{ G}_{T}\bigl(W^{\alpha ,\tilde{\alpha }}\bigr)= \int _{0}^{T}
\boldsymbol{ g}(t)\,dW^{\alpha ,\tilde{\alpha }}_{t}= \Biggl(\int_{0}^{T}
g_{1}(t) \,dW^{\alpha ,\tilde{\alpha }}_{t},\ldots ,\int_{0}^{T} g_{n}(t) \,dW^{
\alpha ,\tilde{\alpha }}_{t} \Biggr)^{\intercal },
\end{equation*}
where $\boldsymbol{ g} =(g_{1},\ldots ,g_{n})^{\intercal }$ is a vectorial
function and $g_{k}\in {\mathbb{L}}^{2}[0,T]$, for $k=1,\ldots ,n$. We
assume, without any loss of generality, that the functions $g_{i}$,
$i=1,\ldots ,n$, are linearly independent. The linearly dependent components
of $\boldsymbol{ g}$ can be simply removed from the conditioning. As we said
in the Introduction, the generalized conditioned process $X^{\boldsymbol{
g;x}}$, for $\boldsymbol{x}\in {\mathbb{R}}^{n}$, is the law of the Gaussian
process\index{Gaussian process} $X$ conditioned on the set
\begin{equation*}
\Biggl\{ \int_{0}^{T} \boldsymbol{ g}(t)\,dW^{\alpha ,\tilde{\alpha }}_{t}=
\boldsymbol{x} \Biggr\}=\bigcap _{i=1}^{n} \Biggl\{ \int _{0}^{T}
g_{i}(t)\,dW^{ \alpha ,\tilde{\alpha }}_{t}=x_{i} \Biggr\}.
\end{equation*}
The law ${\mathbb{P}}^{\boldsymbol{ g;x}}$ of $X^{\boldsymbol{ g;x}}$ is the
regular conditional distribution on $C[0,+\infty )$, endowed with the
topology induced by the sup-norm on compact sets,
\begin{equation*}
{\mathbb{P}}^{\boldsymbol{g;x}}(X\in E)={\mathbb{P}}\bigl(X^{ \boldsymbol{
g;x}}\in E \bigr)={\mathbb{P}} \Biggl(X\in E \Bigm|\int_{0}^{T} \boldsymbol{
g}(t)\,dW^{\alpha ,\tilde{\alpha }}_{t}=\boldsymbol{ x} \Biggr).
\end{equation*}
For more details about existence of such regular conditional distribution
see, for example, \cite{LaG}.

Denote by $C^{\boldsymbol{ g}}=(c^{g_{i}g_{j}}_{ij})_{i,j=1,\ldots ,n}$ the
matrix defined by
\begin{equation*}
c^{g_{i}g_{j}}_{ij}=\operatorname{Cov} \Biggl( \int _{0}^{T} g_{i}(t) \,dW^{
\alpha ,\tilde{\alpha }}_{t},\int_{0}^{T} g_{j}(t) \,dW^{\alpha ,
\tilde{\alpha }}_{t} \Biggr)=\bigl(\alpha ^{2}+\tilde{\alpha
}^{2}\bigr)\int_{0}^{T} g_{i}(t)g_{j}(t)\,dt.
\end{equation*}

The matrix $C^{\boldsymbol{ g}}$ is invertible (since the functions $g_{i}$,
$i=1,\ldots ,n$, are linearly independent). Let us denote
\begin{equation*}
r_{i}^{{ g_{i}}}(t)=\operatorname{Cov} \Biggl(X_{t}, \int_{0}^{T} g_{i}(u)
\,dW^{ \alpha ,\tilde{\alpha }}_{u} \Biggr)=\alpha \int_{0}^{t\wedge T}
K(t,u) g_{i}(u)\,du,
\end{equation*}
and
\begin{equation*}
r^{\boldsymbol{ g}}(t)=\bigl(r_{1}^{g_{1}}(t),\ldots
,r_{n}^{g_{n}}(t)\bigr)^{\intercal }.
\end{equation*}
The following theorem, similar to Theorem 3.1 in \cite{Sot-Yaz}, gives mean
and covariance function of the generalized conditioned process.

\begin{theorem}
The generalized conditioned process $X^{\boldsymbol{ g;x}}$ can be represented
as
\begin{equation*}
X^{\boldsymbol{ g;x}}_{t}=X_{t} - r^{\boldsymbol{ g}}(t)^{\intercal }
\bigl(C^{ \boldsymbol{ g}}\bigr)^{-1} \Biggl( \int_{0}^{T} \boldsymbol{
g}(u)\,dW^{ \alpha ,\tilde{\alpha }}_{u}-\boldsymbol{ x} \Biggr).
\end{equation*}
Moreover, the conditioned process $X^{\boldsymbol{g;x}}$ is a Gaussian
process\index{Gaussian process} with mean
%
\begin{equation}\label{eqn:mean-cond1}
m^{\boldsymbol{ g;x}}(t)={\mathbb{E}}\bigl[X^{\boldsymbol{ g;x}}_{t} \bigr]
=r^{ \boldsymbol{ g}} (t)^{\intercal }\bigl(C^{\boldsymbol{ g}} \bigr)^{-1}
\boldsymbol{ x},
\end{equation}
and covariance
%
\begin{equation}\label{eqn:cov-cond1}
k^{\boldsymbol{g}}(t,s)=\operatorname{Cov}\bigl(X^{\boldsymbol{ g;x}}_{t},X^{
\boldsymbol{ g;x}}_{s} \bigr) =k(t,s)-\kappa ^{\boldsymbol{g}}(t,s),
\end{equation}
where
%
\begin{equation}\label{eqn:kappa-g}
\kappa ^{\boldsymbol{g}}(t,s)=r^{\boldsymbol{ g}} (t)^{\intercal }\bigl(C^{
\boldsymbol{ g}}\bigr)^{-1}r^{\boldsymbol{ g}}(s).
\end{equation}
\end{theorem}

\begin{proof}
It is a classical result on conditioned Gaussian laws. See, e.g., Chapter II,
\S 13, in \cite{Shi}.
\end{proof}

\begin{remark}
Let us note that the covariance function of the conditioned process depends
on the conditioning functions $g_{1},\ldots ,g_{n}$ and on the time $T$, but
not on the vector ${\boldsymbol{x}}$.
\end{remark}

\begin{remark}
If the conditioning functions $g_{i}$ are the indicator functions of the
interval $[0,T_{i})$, for $i=1,\ldots ,n$, then the process is conditioned to
the position of the noisy Brownian motion\index{Brownian motion} at the times $T_{1},\ldots ,
T_{n}$, more precisely to the set $\bigcap_{i=1}^{n} \{ W^{\alpha
,\tilde{\alpha }}_{T_{i}}=x_{i}\}$.
\end{remark}

\begin{remark}
If the conditioning functions are $g_{i}(s)=K(T_{i},s) \mbox{\large
1}_{[0,T_{i})}(s)$, for $i=1,\ldots ,n$, and $\alpha =1$, $\tilde{\alpha
}=0$, then the process is conditioned to its position at the times
$T_{1},\ldots , T_{n}$, more precisely to the set $\bigcap_{i=1}^{n} \{
X_{T_{i}}=x_{i}\}$ (this is a particular case of the conditioned process in
\cite{Pac}).
\end{remark}

\subsection{Large deviations}\label{sec3.2}

Let $\gamma _{\varepsilon }>0$ be an infinitesimal function, i.e. $\gamma
_{\varepsilon }\to 0$ for $\varepsilon \to 0$. In this section
$(X_{t})_{t\geq 0}$ is a continuous Volterra process\index{Volterra process} as in
(\ref{eqn:integral-representation}). Now, in order to achieve a large
deviation principle for the family of processes $((X^{\boldsymbol{
g;x}}_{T+\varepsilon t}- X^{\boldsymbol{ g;x}}_{T})_{t \in
[0,1]})_{\varepsilon >0}$, we have to investigate the behavior of the
functions $k^{\boldsymbol{g}}$ and $m^{\boldsymbol{ g;x}}$ (defined in
(\ref{eqn:cov-cond1}) and (\ref{eqn:mean-cond1}), respectively) in a small
time interval of length $\varepsilon $.

Now we give some conditions on the original process in order to guarantee
that the hypotheses of Theorem~\ref{th:ldp-gaussian} hold for the conditioned
process. The next assumption (Assumption~\ref{ass:limit-cov-cond1}) implies the
existence of a limit covariance.

\begin{assumption}\label{ass:limit-cov-cond1}
For any $T>0$ and for $g_{i}\in {\mathbb{L}}^{2}[0,T]$, $i=1,\ldots ,n$,
there exists a vectorial function
$\bar{r}^{\boldsymbol{g}}=(\bar{r}_{1}^{g_{1}},\ldots ,\bar{r}_{n}^{g_{n}}
)$, possibly $\bar{r}_{i}^{g_{i}}=0$ for some $i=1,\ldots ,n$, such that
%
\begin{equation}\label{eqn:r-bar}
\bar{r}_{i}^{g_{i}}(t)=\lim _{\varepsilon \to 0} \frac{\operatorname{Cov}
(X_{T+\varepsilon t}-X_{T}, \int_{0}^{T} g_{i}(u)\,dW^{\alpha ,\tilde{\alpha
}}_{u}  )}{\gamma _{\varepsilon }}= \lim_{\varepsilon \to 0}
\frac{r_{i}^{g_{i}}(T+\varepsilon t)-r_{i}^{g_{i}}(T)}{\gamma
_{\varepsilon}},
\end{equation}
uniformly in $t\in [0,1]$.
\end{assumption}

The next assumption (Assumption~\ref{ass:exp-tight-cond1}) implies the
exponential tightness\index{exponential tightness} of the family of the centered processes.

\begin{assumption}\label{ass:exp-tight-cond1}
For any fixed $T>0$ there exist constants $M,\hat{\tau }>0$, such that for
$i=1,\ldots ,n$ and $\varepsilon >0$,
\begin{align*}
&\sup_{s,t\in [0,1],s\neq t}\frac{  |{\operatorname{Cov}}  (X_{T+\varepsilon
t}-X_{T+\varepsilon s}, \int_{0}^{T} g_{i}(u)\,dW^{\alpha ,\tilde{\alpha
}}_{u}  )  |}{\gamma _{\varepsilon }|t-s|^{\hat{\tau }}} \\
&\quad=\sup_{s,t\in [0,1],s\neq t} \frac{|r_{i}^{g_{i}}(T+\varepsilon
t)-r_{i}^{g_{i}}(T+\varepsilon s)|}{\gamma _{\varepsilon }|t-s|^{\hat{\tau
}}} \leq M.
\end{align*}
\end{assumption}

\begin{remark}
Let us observe that Assumption~\ref{ass:limit-cov-cond1} implies that for any
fixed $T>0$
\begin{equation*}
\lim_{\varepsilon \to 0}r_{i}^{g_{i}}(T+\varepsilon t)-r_{i}^{g_{i}}(T)=0,
\end{equation*}
uniformly in $t\in [0,1]$. Therefore,
%
\begin{equation}\label{eqn:mean-limit-cond1}
\lim_{\varepsilon \to 0} m^{\boldsymbol{ g;x}}(T+ \varepsilon t)=m^{
\boldsymbol{ g;x}}(T),
\end{equation}
uniformly in $t\in [0,1]$. In fact, one has
\begin{equation*}
m^{\boldsymbol{ g;x}}(T+\varepsilon t)-m^{\boldsymbol{ g;x}}(T)=\bigl(
\boldsymbol{r}^{\boldsymbol{g}}(T+\varepsilon t)-\boldsymbol{r}^{
\boldsymbol{g}}(T) \bigr)^{\intercal }\bigl(C^{\boldsymbol{ g}}\bigr)^{-1} x,
\end{equation*}
and (\ref{eqn:mean-limit-cond1}) immediately follows.
\end{remark}

\begin{remark}\label{rem:condition-exp-tight-cond1}
Let us observe that Assumption \ref{ass:exp-tight-cond1} implies that there
exists $M>0$ such that the following estimate holds for the function $\kappa
^{\boldsymbol{g}}$ defined in (\ref{eqn:kappa-g}):
%
\begin{align}
&\sup_{s,t\in [0,1],s\neq t}\frac{|\kappa ^{\boldsymbol{g}}(T+\varepsilon t,
T+\varepsilon t)-2 \kappa ^{\boldsymbol{g}}(T+\varepsilon t, T+\varepsilon s
) +\kappa ^{\boldsymbol{g}}(T+\varepsilon s , T+\varepsilon s)|}{\gamma
_{\varepsilon }^{2}|t-s|^{2\hat{\tau }}} \nonumber\\
&\quad\leq M.\label{eqn:var-condition-exp-tight-cond1}
\end{align}

In fact, straightforward computations show that
\begin{align*}
&\kappa ^{\boldsymbol{g}}(T+\varepsilon t, T+\varepsilon t) -2
\kappa ^{\boldsymbol{g}}(T+\varepsilon t, T+\varepsilon s ) +\kappa
^{\boldsymbol{g}}(T+\varepsilon s , T+\varepsilon s)\\
&\quad=\bigl(\bigl( \boldsymbol{r}^{\boldsymbol{g}}(T+\varepsilon
t)-\boldsymbol{r}^{\boldsymbol{g}}(T+ \varepsilon s)\bigr)\bigr)^{\intercal
}\bigl(C^{\boldsymbol{ g}}
\bigr)^{-1}\bigl(\bigl(\boldsymbol{r}^{\boldsymbol{g}}(T+\varepsilon
t)-\boldsymbol{r}^{\boldsymbol{g}}(T+\varepsilon s)\bigr)\bigr).
\end{align*}
\end{remark}

Therefore (\ref{eqn:var-condition-exp-tight-cond1}) immediately follows from
Assumption \ref{ass:exp-tight-cond1}.

\begin{proposition}\label{prop:limit-cov}
Under Assumptions \ref{ass:limit-cov} and \ref{ass:limit-cov-cond1}, one has
\begin{equation*}
\lim_{\varepsilon \to 0}
\frac{\operatorname{Cov}(X^{\boldsymbol{g};x}_{T+\varepsilon
t}-X^{\boldsymbol{g};x}_{T}, X^{\boldsymbol{g};x}_{T+\varepsilon
s}-X^{\boldsymbol{g};x}_{T})}{\gamma ^{2}_{\varepsilon }}=
\bar{k}^{\boldsymbol{g}}(t,s),
\end{equation*}%
uniformly in $(t,s)\in [0,1]\times [0,1]$, with
%
\begin{equation}\label{eqn:kg-bar}
\bar{k}^{\boldsymbol{g}}(t,s)=\bar{k}(t,s)-\bar{ \boldsymbol{r}}^{
\boldsymbol{g}}(t)^{\intercal }\bigl(C^{\boldsymbol{ g}} \bigr)^{-1}
\bar{\boldsymbol{ r}}^{\boldsymbol{ g}}(s),
\end{equation}
where $\bar{\boldsymbol{r}}^{\boldsymbol{ g}}(t)^{\intercal
}=(\bar{r}^{g_{1}}_{1}(t), \ldots ,\bar{r}^{g_{n}}_{n}(t))$ and
$\bar{r}_{i}^{g_{i}}(t)$ is defined in (\ref{eqn:r-bar}) for $i=1,\ldots ,n$.
\end{proposition}

\begin{proof}
Taking into account equation (\ref{eqn:cov-cond1}), simple computations show
that for $s,t\in [0,1]$,
%
\begin{align}
&{\operatorname{Cov}}\bigl( X^{\boldsymbol{ g;x}}_{T+\varepsilon t}-
X^{
\boldsymbol{ g;x}}_{T}, X^{\boldsymbol{ g;x}}_{T+\varepsilon s}-
X^{
\boldsymbol{ g;x}}_{T}\bigr)
\nonumber
\\
&\quad= \bigl(k(T+\varepsilon t, T+\varepsilon s) -k(T+ \varepsilon t, T)-k(T+
\varepsilon s , T) +k(T, T)\bigr)+
\nonumber
\\
&\qquad{}-\bigl(\bigl( \bar{\boldsymbol{r}}^{\boldsymbol{ g}}(T+ \varepsilon t)-
\bar{\boldsymbol{r}}^{\boldsymbol{ g}}(T)\bigr)\bigr)^{\intercal }
\bigl(C^{\boldsymbol{ g}}\bigr)^{-1}\bigl(\bigl(\bar{
\boldsymbol{r}}^{
\boldsymbol{ g}}(T+\varepsilon s)-\bar{\boldsymbol{r}}^{
\boldsymbol{ g}}(T)
\bigr)\bigr).\label{eqn:covariance-conditional2-process-epsilon} %
\end{align}

Therefore the claim easily follows from Assumptions \ref{ass:limit-cov} and
\ref{ass:limit-cov-cond1}.
\end{proof}

\begin{remark}%
Notice that $\bar{k}^{\boldsymbol{g}}$ is a continuous covariance function,\index{continuous covariance function}
being the (uniform) limit of continuous, symmetric and positive definite
functions.
\end{remark}

\begin{proposition}\label{prop:limit-cov-cond1}
Under Assumptions \ref{ass:exp-tight} and \ref{ass:exp-tight-cond1} the
family $((X^{\boldsymbol{g;x}}_{T+\varepsilon t}- X^{\boldsymbol{g;x}}_{T}-
{\mathbb{E}}[X^{\boldsymbol{g;x}}_{T+\varepsilon t}- X^{
\boldsymbol{g;x}}_{T}])_{t\in [0,1]})_{\varepsilon >0}$ is exponentially
tight at 
the inverse speed function $\gamma _{\varepsilon
}^{2}$.
\end{proposition}

\begin{proof}
As $((X^{\boldsymbol{g;x}}_{T+\varepsilon t}- X^{\boldsymbol{g;x}}_{T}-
{\mathbb{E}}[X^{\boldsymbol{g;x}}_{T+\varepsilon t}-
X^{\boldsymbol{g;x}}_{T}])_{t\in [0,1]})_{\varepsilon >0}$ is a family of
centered processes, it is enough to prove that
(\ref{eqn:cov-condition-exp-tight}) is satisfied with 
an appropriate speed
function. For $\varepsilon >0$ the covariance of such process is given by
(\ref{eqn:covariance-conditional2-process-epsilon}). Therefore
\begin{align*}
&{\operatorname{Var}}\bigl(X^{\boldsymbol{g;x}}_{T+\varepsilon t}-
X^{
\boldsymbol{g;x}}_{T+\varepsilon s}\bigr)
\\
&\quad = k(T+\varepsilon t, T+\varepsilon t) -2 k(T+ \varepsilon t, T+
\varepsilon s ) +k(T+\varepsilon s , T+\varepsilon s)
\\
&\qquad{}+ \boldsymbol{r}^{\boldsymbol{g}}(T+\varepsilon t)-
\boldsymbol{r}^{\boldsymbol{g}}(T)))^{\intercal }\bigl(C^{\boldsymbol{ g}}
\bigr)^{-1}(\bigl( \boldsymbol{r}^{\boldsymbol{g}}(T+\varepsilon s)-
\boldsymbol{r}^{
\boldsymbol{g}}(T)\bigr).
\end{align*}
From Assumption \ref{ass:exp-tight} we already know that
\begin{equation*}
\sup_{s,t\in [0,1],s\neq t} \frac{k(T+\varepsilon t, T+\varepsilon t) -2
k(T+\varepsilon t, T+\varepsilon s ) +k(T+\varepsilon s , T+\varepsilon
s)}{\gamma _{\varepsilon }^{2} \,|t-s|^{2\tau }} \leq M.
\end{equation*}
Furthermore, Assumption \ref{ass:exp-tight-cond1} implies that
\begin{equation*}
\sup_{s,t\in [0,1],s\neq t}
\frac{|(\boldsymbol{r}^{\boldsymbol{g}}(T+\varepsilon
t)-\boldsymbol{r}^{\boldsymbol{g}}(T)))^{\intercal }(C^{\boldsymbol{
g}})^{-1}((\boldsymbol{r}^{\boldsymbol{g}}(T+\varepsilon
s)-\boldsymbol{r}^{\boldsymbol{g}}(T))|}{\gamma _{\varepsilon }^{2}
\,|t-s|^{2\hat{\tau }}} \leq M.
\end{equation*}

Therefore condition (\ref{eqn:cov-condition-exp-tight}) holds with the
inverse speed $\eta _{\varepsilon }=\gamma _{\varepsilon }^{2}$ and $\beta
=\tau \wedge \hat{\tau }$.
\end{proof}

We are now ready to prove the main large deviation result of this section.

\begin{theorem}\label{th:ldp-cond1}
Suppose $(X_{t})_{t\geq 0}$ satisfies Assumptions \ref{ass:limit-cov},
\ref{ass:exp-tight}, \ref{ass:limit-cov-cond1} and \ref{ass:exp-tight-cond1}.
Suppose, furthermore, that the (existing) covariance function
$\bar{k}^{\boldsymbol{g}}$ defined in Proposition \ref{prop:limit-cov} is
regular enough, then the family of processes
$((X^{\boldsymbol{g;x}}_{T+\varepsilon t} - X^{\boldsymbol{g;x}}_{T})_{t \in
[0,1]})_{\varepsilon >0}$ satisfies a large deviation principle on $C[0,1]$
with the inverse speed $\gamma ^{2}_{\varepsilon }$ and the good rate
function
%
\begin{equation}\label{eqn:rate-cond1}
J_{X}^{\boldsymbol{g}}(h)=
\begin{cases}
\frac{1}{2}\, \|h\|^{2}_{\bar{{\mathscr{H}}}^{\boldsymbol{g}}}, &  h\in \bar{{\mathscr{H}}}^{\boldsymbol{g}},\\
+\infty, & \text{otherwise},
\end{cases}
\end{equation}
where $\bar{{\mathscr{H}}}^{\boldsymbol{g}}$ is the reproducing kernel\index{reproducing kernel Hilbert space}
Hilbert space associated to the covariance function
$\bar{k}^{\boldsymbol{g}}$.
\end{theorem}

\begin{proof}
Cosider the family of centered processes
$((X^{\boldsymbol{g;x}}_{T+\varepsilon t}- X^{\boldsymbol{g;x}}_{T}-
{\mathbb{E}}[X^{\boldsymbol{g;x}}_{T+\varepsilon t}- X^{
\boldsymbol{g;x}}_{T}])_{t\in [0,1]})_{\varepsilon >0}$. Thanks to
Proposition~\ref{prop:limit-cov-cond1} this family of processes is
exponentially tight at the inverse speed $\gamma _{\varepsilon }^{2}$. Thanks
to Proposition~\ref{prop:limit-cov}, for any $\lambda \in
{\mathscr{M}}[0,1]$, one has
\begin{align*}
&\lim_{\varepsilon \to 0} \frac{\operatorname{Var}(\langle \lambda , X^{\boldsymbol{g;x}}_{T+\varepsilon \cdot }-X^{\boldsymbol{g;x}}_{T}\rangle )}{\gamma ^{2}_{\varepsilon }}\\
&\quad=\lim _{\varepsilon \to 0} \int_{0}^{1}\,d \lambda
(v)\int_{0}^{1}\,d\lambda (u)
\frac{\operatorname{Cov}(X^{\boldsymbol{g;x}}_{T+\varepsilon
v}-X^{\boldsymbol{g;x}}_{T}, X^{\boldsymbol{g;x}}_{T+\varepsilon
u}-X^{\boldsymbol{g;x}}_{T})}{\gamma ^{2}_{\varepsilon }} \\
&\quad=\int_{0}^{1}\,d\lambda (v)\int_{0}^{1}\,d\lambda (u)
\bar{k}^{\boldsymbol{g}}(v,u),
\end{align*}
where $\bar{k}^{\boldsymbol{g}}$ is defined in (\ref{eqn:kg-bar}). Since
$\bar{k}^{\boldsymbol{g}}$ is the covariance function of a continuous
Volterra process,\index{Volterra process} a large deviation principle for
$((X^{\boldsymbol{g;x}}_{T+\varepsilon t} - X^{\boldsymbol{g;x}}_{T}-
{\mathbb{E}}[X^{\boldsymbol{g;x}}_{T+\varepsilon t}-\break X^{
\boldsymbol{g;x}}_{T}])_{t\in [0,1]})_{\varepsilon >0}$ actually holds from
Theorem~\ref{th:ldp-gaussian} with the inverse speed $\gamma _{\varepsilon
}^{2}$ and the good rate function given by (\ref{eqn:rate-cond1}). From
Equation~(\ref{eqn:mean-limit-cond1}) and Remark~\ref{rem:exp-equiv} the same
large deviation principle holds for the noncentered family
$((X^{\boldsymbol{g;x}}_{T+\varepsilon t} - X^{\boldsymbol{g;x}}_{T})_{t \in
[0,1]})_{\varepsilon >0}$.
\end{proof}

\subsection{Examples}\label{sect:examples}

In this section we consider some examples to which Theorem~\ref{th:ldp-cond1}
applies. Therefore we want to verify that Assumptions~\ref{ass:limit-cov},
\ref{ass:exp-tight}, \ref{ass:limit-cov-cond1} and~\ref{ass:exp-tight-cond1}
are fulfilled. Let $X$ be a continuous, centered Volterra process process
with kernel $K$. Suppose $g_{1}(t)=\mbox{\large 1}_{[0,T)}(t)$ and
$g_{2}(t)=\frac{T-t}{T}\mbox{\large 1}_{[0,T)}(t)$, that is,
\begin{equation*}
W^{\alpha ,\tilde{\alpha }}_{T}=\int_{0}^{T} g_{1}(u)\,dW^{\alpha ,
\tilde{\alpha }}_{u}=x_{1}
\end{equation*}
and by the integration by parts formula,
\begin{equation*}
\frac{1}{T} \int_{0}^{T} W^{\alpha ,\tilde{\alpha }}_{u} \,du= \int_{0}^{T}
g_{2}(u) \,dW^{\alpha ,\tilde{\alpha }}_{u} = x_{2}.
\end{equation*}
Then the matrix $ (C^{\boldsymbol{ g}})^{-1}$ is given by
\begin{equation*}
\bigl(C^{\boldsymbol{ g}}\bigr)^{-1}=\frac{1}{\det(C^{\boldsymbol{ g}})} \lleft (
\begin{array}{rr}
c^{ g_{2}g_{2}}_{22} & -c^{ g_{1}g_{2}}_{12}
\\[5pt]
-c^{ g_{1}g_{2}}_{12} &c^{g_{1}g_{1}}_{11}%
\end{array} %
\rright ),
\end{equation*}
where
\begin{align*}
c^{g_{1}g_{1}}_{11}&= \bigl(\alpha ^{2} + \tilde{ \alpha }^{2}\bigr) T,
\qquad c^{ g_{1}g_{2}}_{12}= \bigl(\alpha ^{2} + \tilde{\alpha
}^{2}\bigr) \frac{T}{2},
\\
c^{ g_{2}g_{2}}_{22}& =\bigl(\alpha ^{2} + \tilde{
\alpha }^{2}\bigr)\frac{T}{3}, \qquad \det
\bigl(C^{\boldsymbol{ g}}\bigr)=\bigl(\alpha ^{2} + \tilde{\alpha
}^{2}\bigr)^{2} \frac{ T^{2}}{12}.
\end{align*}

\begin{example}[Fractional Brownian Motion]\label{ex:FBMcond1}
Let $X$ be the fractional Brownian motion\index{fractional Brownian motion} of the Hurst index $H>1/2$. The
fractional Brownian motion\index{fractional Brownian motion} with the Hurst parameter $H \in (0,1)$ is the centered
Gaussian process with covariance function
\begin{equation*}
k(t,s)=\frac{1}{2} \bigl(t^{2H}+s^{2H}-|t-s|^{2H} \bigr).
\end{equation*}
The fractional Brownian motion\index{fractional Brownian motion} is a Volterra process\index{Volterra process} with kernel, for $s\leq
t$,
%
\begin{equation}
\label{eqn:kernel fbm}%
K(t,s) = c_{H} \Biggl[ \biggl(
\frac{t}{s} (t-s) \biggr) ^{H-1/2} - \biggl( H-
\frac{1}{2} \biggr) s^{1/2 - H} \int_{s}^{t}
\!u^{H-3/2}(u-s)^{H-1/2} \,du \Biggr],
\end{equation}
where $c_{H} =   ( \frac{2H \, \Gamma (3/2-H)}{\Gamma (H+1/2) \, \Gamma
(2-2H)}   ) ^{1/2}$. Notice that when $H=1/2$ we have $K(t,s) =
\textbf{1}_{[0,t]}(s)$, and then the fractional Brownian motion\index{fractional Brownian motion} reduces to
the Wiener process.

First, let us prove that there exists a limit covariance and that it is
regular enough. For $s\leq t$, one has
\begin{equation*}
\frac{\operatorname{Cov}(X_{T+\varepsilon t}-X_{T}, X_{T+\varepsilon
s}-X_{T})}{\varepsilon ^{2H}} =\operatorname{Cov}(X_{t}, X_{s}),
\end{equation*}
because of the homogeneity and self-similarity properties holding for the
fractional Brownian motion,\index{fractional Brownian motion} so that the limit in (\ref{eqn:k-bar}) trivially
exists and Assumption~\ref{ass:limit-cov} holds with $\bar{k}(t,s)=k(t,s)$
and $\gamma _{\varepsilon }=\varepsilon ^{H}$. Now let us prove that
Assumption~\ref{ass:limit-cov-cond1} is fulfilled.
\begin{align*}
&{\operatorname{Cov}} \bigl(X_{T+\varepsilon t}-X_{T},
W^{\alpha ,\tilde{\alpha }}_{T} \bigr)
\\
&= \int_{0}^{T} \bigl(K(T+\varepsilon t,
u)- K(T,u) \bigr)\,du
\\
&= c_{H} \int_{0}^{T} \biggl(
\biggl(\frac{T+\varepsilon t}{u} \biggr)^{H-1/2}(T+ \varepsilon t
-u)^{H-1/2}- \biggl(\frac{T}{u} \biggr)^{H-1/2}(T-u)^{H-1/2}
\biggr) \,du+
\\
&\quad{}- c_{H} (H-1/2) \int_{0}^{T}
\frac{1}{u^{H-1/2}}\int_{T}^{T+
\varepsilon t}
(v-u)^{H-1/2} v^{H-3/2} \,dv \, du.
\end{align*}

Thanks to the Lagrange theorem,\index{Lagrange theorem}
{\small
\begin{align*}
& \biggl( \biggl(\frac{T+\varepsilon t}{u} \biggr)^{H-1/2}(T+
\varepsilon t -u)^{H-1/2}- \biggl(\frac{T}{u}
\biggr)^{H-1/2}(T-u)^{H-1/2} \biggr)
\\
&= \frac{H-1/2}{u^{H-1/2}}
\bigl[ ( T+ \xi _{\varepsilon })^{H-3/2} (T+\xi _{\varepsilon }-u)^{H-1/2}+
( T+ \xi _{\varepsilon })^{H-1/2} (T+\xi _{\varepsilon }-u)^{H-3/2}
\bigr]\varepsilon t,
\end{align*}}%
for $\xi _{\varepsilon }\in [0,\varepsilon t]$. Therefore from the Lebesgue
theorem,
\begin{align*}
&\lim_{\varepsilon \to 0}\frac{1}{\varepsilon }\int
_{0}^{T} \biggl( \biggl(\frac{T+\varepsilon t}{u}
\biggr)^{H-1/2}(T+\varepsilon t -u)^{H-1/2}- \biggl(
\frac{T}{u} \biggr)^{H-1/2}(T-u)^{H-1/2} \biggr)\,du
\\*
&= t\,(H-1/2)\int_{0}^{T} \frac{1}{u^{H-1/2}}
\bigl(T^{H-3/2}(T -u)^{H-1/2}+ T^{H-1/2}(T-u)^{H-3/2}
\bigr)\,du ,
\end{align*}
uniformly in $t\in [0,1]$. Furthermore, in a similar way, we have,
\begin{align*}
&\lim_{\varepsilon \to 0}\frac{1}{\varepsilon }\int
_{0}^{T}\frac{1}{u^{H-1/2}}\int _{T}^{T+\varepsilon t} (v-u)^{H-1/2}
v^{H-3/2} dv \, du
\\
&\quad=t\,\int_{0}^{T} \frac{1}{u^{H-1/2}} (T-u)^{H-1/2} T^{H-3/2} \,du.
\end{align*}
Therefore,
\begin{equation*}
\bar{r}_{1}^{g_{1}}(t)=\lim_{\varepsilon \to 0} \frac{\operatorname{Cov}
(X_{T+\varepsilon t}-X_{T}, W^{\alpha ,\tilde{\alpha }}_{T}  )}{\varepsilon
^{H}}=0,\vadjust{\eject}
\end{equation*}
uniformly in $t\in [0,1]$. Similar calculations show that
\begin{equation*}
\bar{r}_{2}^{g_{2}}(t)=\lim_{\varepsilon \to 0} \frac{\operatorname{Cov}
(X_{T+\varepsilon t}-X_{T}, \int_{0}^{T} \frac{T-u}{T} \,dW^{\alpha
,\tilde{\alpha }}_{u}  )}{\varepsilon ^{H}}=0\xch{.}{,}
\end{equation*}
So, we have $\bar{k}^{\boldsymbol{g}}(t,s)= k(t,s)$ and therefore the limit
covariance exists and is regular enough.

Now let us prove the exponential tightness\index{exponential tightness} of the family of processes. For
$s< t$,
\begin{eqnarray*}
{\operatorname{Var}}(X_{T+\varepsilon t}-X_{T+\varepsilon s})=\varepsilon
^{2H}(t-s)^{2H},
\end{eqnarray*}
then Assumption \ref{ass:exp-tight} holds with $\tau =H$ and
$\gamma _{\varepsilon }=\varepsilon ^{H}$. For $s<t$, we have
{\small
\begin{align*}
&{\operatorname{Cov}} \bigl(X_{T+\varepsilon t}-X_{T+\varepsilon s},
W^{\alpha ,
\tilde{\alpha }}_{T} \bigr)\\
&= \int_{0}^{T}
\bigl(K(T+\varepsilon t, u)- K(T+ \varepsilon s ,u) \bigr)\,du
\\
&= c_{H}\int_{0}^{T}
\biggl( \biggl(\frac{T+\varepsilon t}{u} \biggr)^{H-1/2}\!\!\!\!\!  (T+
\varepsilon t -u)^{H-1/2}\,{-}\, \biggl( \frac{T+\varepsilon t}{u}
\biggr)^{H-1/2}\!\!\!\!\! (T+ \varepsilon s -u)^{H-1/2}
\biggr) \,du+
\\
&\quad{}- c_{H} (H-1/2) \int_{0}^{T}
\frac{1}{u^{H-1/2}}\int_{T+\varepsilon s}^{T+
\varepsilon t}
(v-u)^{H-1/2} v^{H-3/2} \,dv \, du.%
\end{align*}}%

Thanks to the Lagrange theorem\index{Lagrange theorem} we can find $M>0$ such that
{\small
\begin{align*}
&\int_{0}^{T} \biggl( \biggl(
\frac{T+\varepsilon t}{u} \biggr)^{H-1/2}(T+\varepsilon t
-u)^{H-1/2}- \biggl(\frac{T+\varepsilon s}{u} \biggr)^{H-1/2}(T+
\varepsilon s -u)^{H-1/2} \biggr)\,du
\\
&\leq \varepsilon (t-s)\! \biggl(H- \frac{1}{2}
\biggr)\!\!\int_{0}^{T}\!\!\!\frac{1}{u^{H-\frac{1}{2}}} \bigl[
T^{H-\frac{3}{2}} (T+1 -u)^{H-\frac{1}{2}}\!\!+ ( T+ 1)^{H-\frac{1}{2}} (T
-u)^{H-\frac{3}{2}}\! \bigr]\,du \\
&\leq \varepsilon M (t-s)
\end{align*}}%
and
\begin{align*}
&\int_{0}^{T} \frac{1}{u^{H-1/2}}\int_{T+\varepsilon s}^{T+\varepsilon t}
(v-u)^{H-1/2} v^{H-3/2} \,dv \, du \\
& =(t-s)\,\int _{0}^{T}\frac{1}{u^{H-1/2}} (T+\xi _{\varepsilon }-u)^{H-1/2}
(T+\xi _{\varepsilon })^{H-3/2} \,du\leq \varepsilon M(t-s).
\end{align*}
Therefore, \textit{a fortiori},
\begin{equation*}
\sup_{s,t\in [0,1],s\neq t} \frac{  |{\operatorname{Cov}}  (X_{T+\varepsilon
t}-X_{T+\varepsilon s}, W^{\alpha ,\tilde{\alpha }}_{T}  )  |}{\varepsilon
^{H}|t-s|} \leq M.
\end{equation*}
Similar calculations show that
\begin{equation*}
\sup_{s,t\in [0,1],s\neq t} \frac{  |{\operatorname{Cov}}  (X_{T+\varepsilon
t}-X_{T+\varepsilon s}, \int_{0}^{T}\frac{T-u}{T} \,dW^{\alpha ,\tilde{\alpha
}}_{u}  )  |}{\varepsilon ^{H}|t-s|} \leq M.
\end{equation*}
Thus, Assumption \ref{ass:exp-tight-cond1} is fulfilled with $\hat{\tau }=1$.
Therefore the family $((X^{\boldsymbol{g;x}}_{T+\varepsilon t}-\break
X^{\boldsymbol{g;x}}_{T})_{t \in [0,1]})_{\varepsilon >0}$ satisfies a large
deviation principle with the inverse speed function\vspace{2pt} $\gamma _{\varepsilon
}^{2}=\varepsilon ^{2H}$ as the
nonconditioned process. Note that the same
result was obtained in \cite{Car-Pac} for the $n$-fold conditional fractional
Brownian motion.\index{fractional Brownian motion}
\end{example}

\begin{example}[$m$-fold integrated Brownian motion]\label{ex:mIBM-cond1}
For $m\geq 1$, let $X$ be the $m$-fold integrated Brownian motion,\index{Brownian motion} i.e.
\begin{equation*}
X_{t}=\int_{0}^{t}\,du \Biggl(\int _{0}^{u}\,du_{m-1}\cdots \int _{0}^{u_{2}}
\,du_{1} B_{u_{1}} \Biggr).
\end{equation*}
It is a continuous Volterra process\index{Volterra process} with kernel $K(t,u)=\frac{1}{m!}
(t-u)^{m}$ and covariance function
\begin{equation*}
k(t,s)=\frac{1}{(m!)^{2}}\int_{0}^{s\wedge t}(s-\xi )^{m}(t-\xi )^{m} \,d\xi.
\end{equation*}
First, let us prove that there exists a limit covariance and that it is
regular enough. Assumption~\ref{ass:limit-cov} is fulfilled. In fact, for
$s\leq t$, we have
\begin{align*}
&\lim_{\varepsilon \to 0} \frac{\operatorname{Cov}(X_{T+\varepsilon t}-X_{T},
X_{T+\varepsilon s}-X_{T})}{\varepsilon ^{2}}
\\
&= \lim_{\varepsilon \to 0} \frac{1}{(m!)^{2}}
\frac{1}{\varepsilon
^{2}} \int_{T}^{T+\varepsilon s} (T+
\varepsilon t -u)^{m} (T+\varepsilon s -u)^{m} \,du
\\
&\quad{} +\lim_{\varepsilon \to 0} \frac{1}{(m!)^{2}\varepsilon
^{2}}\int
_{0}^{T} \bigl((T\,{+}\, \varepsilon t\,
{-}\,u)^{m} \,{-}\,(T\,{-}\,u)^{m} \bigr) \bigl((T\,{+}\,\varepsilon s
\,{-}\,u)^{m} -(T -u)^{m} \bigr) \,du.
\end{align*}
It is straightforward to show that
\begin{equation*}
\bar{k}(t,s)=\lim_{\varepsilon \to 0}
\frac{\operatorname{Cov}(X_{T+\varepsilon t}-X_{T}, X_{T+\varepsilon
s}-X_{T})}{\varepsilon ^{2}}= \frac{1}{(m!)^{2}} \frac{m^{2}}{2m-1}T^{2m-1}
st,
\end{equation*}
uniformly in $(t,s)\in [0,1]\times [0,1]$. Furthermore,
\begin{eqnarray*}
\bar{r}_{1}^{g_{1}}(t)&=&\lim_{\varepsilon \to 0}
\frac{\operatorname{Cov}  (X_{T+\varepsilon t}-X_{T}, W^{\alpha
,\tilde{\alpha }}_{T}  )}{\varepsilon }
\\
&=& \lim_{\varepsilon \to 0}\frac{\alpha }{m!} \frac{1}{\varepsilon }
\int_{0}^{T} \bigl((T+\varepsilon t -u)^{m} -(T -u)^{m} \bigr) \,du=
\frac{\alpha }{m!} T^{m} \, t,
\end{eqnarray*}
and
\begin{eqnarray*}
\bar{r}_{2}^{g_{2}}(t)&=&\lim_{\varepsilon \to 0}
\frac{\operatorname{Cov}  (X_{T+\varepsilon t}-X_{T}, \int_{0}^{T}
\frac{T-u}{T}\, dW^{\alpha ,\tilde{\alpha }}_{u}  )}{\varepsilon }
\\
&=&\lim_{\varepsilon \to 0}\frac{\alpha }{m!\, T} \frac{1}{\varepsilon }
\int_{0}^{T} \bigl((T+\varepsilon t -u)^{m} -(T -u)^{m} \bigr) (T-u)\,du
\\
&=&\frac{\alpha }{m!}\frac{m}{m+1} T^{m} \, t,
\end{eqnarray*}
uniformly in $t\in [0,1]$. Therefore also Assumption
\ref{ass:limit-cov-cond1} is fulfilled.

Thus, we have $\bar{k}^{\boldsymbol{g}}(t,s)= a \,st$, where
\begin{equation*}
a=\frac{1}{(m!)^{2}} \biggl(\frac{m^{2}}{2m-1} T^{2m-1} - \alpha ^{2} T^{2m}
\biggl(1,\frac{m}{m+1} \biggr) \bigl(C^{\boldsymbol{ g}}\bigr)^{-1} \biggl(1,
\frac{m}{m+1} \biggr)^{\intercal } \biggr).
\end{equation*}
Note that $\bar{k}^{\boldsymbol{g}}$ is regular enough.\eject

Now let us prove the exponential tightness\index{exponential tightness} of the family of processes. For
$s< t$, there exists a constant $M>0$, such that
\begin{align*}
&\operatorname{Var}(X_{T+\varepsilon t}-X_{T+\varepsilon s})\\
&=\frac{1}{(m!)^{2}} \Biggl(\int_{T+\varepsilon s }^{T+\varepsilon t }\!\!\!(T\,{+}\,
\varepsilon t \,{-}\,u)^{2m}\,du\,{+}\, \int_{0}^{T+\varepsilon s}\!\!\!
\bigl((T\,{+}\,\varepsilon t \,{-}\,u)^{m} \,{-}\,(T\,{+}\,\varepsilon s
\,{-}\,u)^{m} \bigr)^{2} \,du \Biggr)\\
&\leq M \varepsilon ^{2} (t-s)^{2}.
\end{align*}
Then Assumption \ref{ass:exp-tight} holds with $\tau =1$ and $\gamma
_{\varepsilon }=\varepsilon $. For $s<t$,
\begin{align*}
\bigl|{\operatorname{Cov}} \bigl(X_{T+\varepsilon t}-X_{T+\varepsilon s},
W^{\alpha ,\tilde{\alpha }}_{T} \bigr) \bigr|&= \int_{0}^{T}
\bigl((T+\varepsilon t-u)^{m}-(T+\varepsilon s-u)^{m}
\bigr)\,du\\
&= \sum_{k=0}^{m-1} { \binom{m}{k}} \frac{1}{k+1}T^{k+1} \varepsilon
^{m-k}(t-s)^{m-k}\xch{.}{,}
\end{align*}
Then we have
\begin{equation*}
\sup_{s,t\in [0,1],s\neq t} \frac{  |{\operatorname{Cov}}  (X_{T+\varepsilon
t}-X_{T+\varepsilon s}, W^{\alpha ,\tilde{\alpha }}_{T}  )  |}{\varepsilon
|t-s|} \leq M.
\end{equation*}
Similar calculations show that
\begin{equation*}
\sup_{s,t\in [0,1],s\neq t} \frac{  |{\operatorname{Cov}}  (X_{T+\varepsilon
t}-X_{T+\varepsilon s}, \int_{0}^{T} \frac{T-u}{T} \,dW^{\alpha
,\tilde{\alpha }}_{u}  )  |}{\varepsilon |t-s|} \leq M.
\end{equation*}
Thus, Assumption \ref{ass:exp-tight-cond1} is fulfilled with $\hat{\tau }=1$.
Therefore the family $((X^{\boldsymbol{g;x}}_{T+\varepsilon t}-\break
X^{\boldsymbol{g;x}}_{T})_{t \in [0,1]})_{\varepsilon >0}$ satisfies a large
deviation principle with the inverse speed $\gamma _{\varepsilon
}^{2}=\varepsilon ^{2}$.
\end{example}

\begin{example}[Integrated Volterra Process]\label{ex:IVP-cond1}
Let $Z$ be a Volterra process\index{Volterra process} with kernel $K$ satisfying condition
(\ref{eqn:mod-cont}) for some $A>0$. Let $X$ be the integrated process, i.e.
\begin{equation*}
X_{t}=\int_{0}^{t} Z_{u} \,du.
\end{equation*}
The process $X$ is a continuous, Volterra process\index{Volterra process} with kernel
\begin{equation*}
h(t,s)=\int_{s}^{t} K(u,s) \,du,\quad \mbox{i.e. } X_{t}=\int_{0}^{t} h(t,s)
\,dB_{s}.
\end{equation*}
First, let us prove that there exists a limit covariance and that it is
regular enough. Assumption \ref{ass:limit-cov} is fulfilled, in fact, for
$s\leq t$, we have
\begin{eqnarray*}
&&\lim_{\varepsilon \to 0} \frac{\operatorname{Cov}(X_{T+\varepsilon
t}-X_{T}, X_{T+\varepsilon s}-X_{T})}{\varepsilon ^{2}}
\\*
&&\quad=\lim_{\varepsilon \to 0} \frac{\int_{T}^{T+\varepsilon s}
h(T+\varepsilon t, u) h(T+\varepsilon s, u) \,du}{\varepsilon ^{2}}
\\*
&&\qquad{}+ \lim_{\varepsilon \to 0} \frac{ \int_{0}^{T} (
h(T+\varepsilon t, u)- h(T) )(h(T+\varepsilon s, u)- h(T,
u))\,du}{\varepsilon ^{2}}.
\end{eqnarray*}
Now, one has
\begin{align*}
&\int_{T}^{T+\varepsilon s} h(T+\varepsilon t,
u) h(T+\varepsilon s, u) \,du\\
&\quad= \varepsilon ^{3} \int _{0}^{s} \Biggl(\int_{u}^{t} K(T+\varepsilon x,
T+\varepsilon u) \,dx\int_{u}^{s} K( T+\varepsilon x, T+\varepsilon u)\,dx
\Biggr)\,du
\end{align*}
and
\begin{align*}
&\int_{0}^{T} \bigl( h(T+
\varepsilon t, u)- h(T,u) \bigr) \bigl(h(T+\varepsilon s, u)- h(T, u)\bigr)\,du
\\
&\quad= \int_{0}^{T} \Biggl(\int_{T}^{T+\varepsilon s } K(v,u)\,dv
\int_{T}^{T+\varepsilon t } K(v,u)\,dv \Biggr)\,du.
\end{align*}
Therefore
\begin{eqnarray*}
\lim_{\varepsilon \to 0} \frac{\operatorname{Cov}(X_{T+\varepsilon
t}-X_{T}, X_{T+\varepsilon s}-X_{T})}{\varepsilon ^{2}}=st\,\int _{0}^{T}
K^{2}(T,u) \,du,
\end{eqnarray*}
uniformly in $(t,s)\in [0,1]\times [0,1]$. Furthermore, with similar
calculations we have
\begin{equation*}
\bar{r}_{1}^{g_{1}}(t)=\lim_{\varepsilon \to 0} \frac{\operatorname{Cov}
(X_{T+\varepsilon t}-X_{T}, W^{\alpha ,\tilde{\alpha }}_{T}  )}{\varepsilon
}= \alpha \, t \int_{0}^{T} K(T,u) \,du,
\end{equation*}
and
\begin{equation*}
\bar{r}_{2}^{g_{2}}(t)=\lim_{\varepsilon \to 0} \frac{\operatorname{Cov}
(X_{T+\varepsilon t}-X_{T}, \int_{0}^{T} \frac{T-u}{T}\, dW^{\alpha
,\tilde{\alpha }}_{u}  )}{\varepsilon }= \alpha \,\frac{t}{T}\int_{0}^{T}
K(T,u) (T-u) \,du,
\end{equation*}
uniformly in $t\in [0,1]$. Therefore also Assumption
\ref{ass:limit-cov-cond1} is fulfilled.

So, we have $\bar{k}^{\boldsymbol{g}}(t,s)= a \,st$, where
\begin{equation*}
a= \int_{0}^{T} K^{2}(T,u) \,du - \alpha ^{2} A^{\intercal }\bigl(C^{
\boldsymbol{ g}} \bigr)^{-1} A,
\end{equation*}
and  $A^{\intercal }=  (\int_{0}^{T} K(T,u) \,du,\frac{1}{T}\int_{0}^{T}
K(T,u)(T-u) \,du  )$. Note that $\bar{k}^{\boldsymbol{g}}$ is regular enough.
Let us now prove the exponential tightness.\index{exponential tightness} We have, for $s< t$,
\begin{align*}
&\operatorname{Var}(X_{T+\varepsilon t}-X_{T+\varepsilon s})\\
&\quad= \int_{T+\varepsilon s}^{T+\varepsilon t} h(T+\varepsilon t, u)^{2} \,du+
\int_{0}^{T+\varepsilon s} \bigl( h(T+\varepsilon t, u)- h(T+\varepsilon s,u)
\bigr)^{2}.
\end{align*}
Now, recalling that $K$ is a square integrable function, there exists a
constant $M>0$, such that
\begin{align*}
\int_{T+\varepsilon s}^{T+\varepsilon t} h(T+\varepsilon t, u)^{2} \,du&=
\varepsilon ^{3} \int_{s}^{t} \Biggl(\int_{u}^{t} K(T+\varepsilon x, T+
\varepsilon u) \,dx \Biggr)^{2}\,du\\*
&\leq \varepsilon ^{3} M \int_{ s}^{ t } (t-u)\,du\leq M\varepsilon
^{3}(t-s)^{2}
\end{align*}
and
\begin{align*}
\int_{0}^{T+\varepsilon s} ( h(T+\varepsilon
t, u)-h(T+\varepsilon s, u)^{2}\,du&= \varepsilon ^{2}\int
_{0}^{T+\varepsilon s} \Biggl(\int_{ s}^{ t }
K(T+\varepsilon v,u)\,dv \Biggr)^{2}\,du\\
&\leq M\varepsilon
^{2} (t-s)^{2},
\end{align*}
therefore Assumption \ref{ass:exp-tight} holds with $\tau =1$ and $\gamma
_{\varepsilon }=\varepsilon $.

With similar computations we can prove that also
Assumption~\ref{ass:exp-tight-cond1} is fulfilled with $\hat{\tau }=1$ and
$\gamma _{\varepsilon }=\varepsilon $. For $s<t$, we have
\begin{align*}
\bigl|{\operatorname{Cov}} \bigl(X_{T+\varepsilon t}-X_{T+\varepsilon s},
W^{\alpha ,\tilde{\alpha }}_{T} \bigr) \bigr| &= \biggl|\int_{0}^{T}
\bigl( h(T+\varepsilon t,u)-h(T+\varepsilon s)\bigr)\,du \biggr|\\
&= \biggl|\int
_{0}^{T} \int_{T+\varepsilon s}^{T+\varepsilon t}
K(v,u) \,dv\, du \biggr| \leq M\varepsilon (t-s).
\end{align*}
Therefore
\begin{equation*}
\sup_{s,t\in [0,1],s\neq t} \frac{  |{\operatorname{Cov}}  (X_{T+\varepsilon
t}-X_{T+\varepsilon s}, W^{\alpha ,\tilde{\alpha }}_{T}  )  |}{\varepsilon
|t-s|} \leq M.
\end{equation*}
Similar calculations show that
\begin{equation*}
\sup_{s,t\in [0,1],s\neq t} \frac{  |{\operatorname{Cov}}  (X_{T+\varepsilon
t}-X_{T+\varepsilon s}, \int_{0}^{T} \frac{T-u}{T} \,dW^{\alpha
,\tilde{\alpha }}_{u}  )  |}{\varepsilon |t-s|} \leq M.
\end{equation*}
Therefore the family $((X^{\boldsymbol{g;x}}_{T+\varepsilon t}-
X^{\boldsymbol{g;x}}_{T})_{t \in [0,1]})_{\varepsilon >0}$ satisfies a large
deviation principle with the inverse speed function $\gamma _{\varepsilon
}^{2}=\varepsilon ^{2}$.
\end{example}

\section{Conditioning to a path}\label{sect:cond2}

\subsection{Conditional law\index{conditional law}}\label{sec4.1}

Let $(\Omega , {\mathscr{F}}, ({\mathscr{F}}_{t})_{t\geq 0}, {\mathbb{P}})$
be a filtered probability space. On this space we consider a Brownian motion\index{Brownian motion}
$ B=( B)_{t\geq 0}$, a continuous real Volterra process $X=(X_{t})_{t\geq 0}$\index{continuous real Volterra process}
and another Brownian motion\index{Brownian motion} $\tilde{B}=(\tilde{B})_{t\geq 0}$ independent of
$B$. Fix $\alpha , \tilde{\alpha }\in {\mathbb{R}}$ and define $W^{\alpha
,\tilde{\alpha }}= \alpha B + \tilde{\alpha }\tilde{B}$.

We are interested in the regular conditional law of the process $X$ given the
$\sigma $-algebra ${\mathscr{F}}^{\alpha ,\tilde{\alpha }}_{T}$, where
$({\mathscr{F}}^{\alpha ,\tilde{\alpha }}_{t})_{t\geq 0}$ is the filtration
generated by the mixed Brownian motion\index{Brownian motion} $W^{\alpha ,\tilde{\alpha }}$, i.e. we
want to condition the process to the past of the mixed Brownian motion\index{Brownian motion} up to
a fixed time $T>0$. To do this, consider the conditional law\index{conditional law} on $C[0,+\infty
)$ endowed with the topology induced by the sup-norm on compact sets,
${\mathbb{P}}(X \in \cdot \mid {\mathscr{F}}^{\alpha , \tilde{\alpha
}}_{T})$. There exists a regular version of such conditional probability (see
\cite{LaG} and \cite{Pfa}), namely a version such that $\Gamma \mapsto
{\mathbb{P}}(X \in \Gamma \mid {\mathscr{F}}^{a,b}_{T})$ is almost surely a
Gaussian probability law.

The following theorem, Theorem 2.1 in \cite{Sot-Vii2}, gives mean and covariance
function of the Gaussian conditional law.\index{Gaussian conditional law}

\begin{theorem}
For $T>0$, the regular conditional law of $X \mid {\mathscr{F}}^{a,b}_{T}$ is a
Gaussian measure with the random mean
\begin{equation*}
\Psi _{t}\bigl(W^{\alpha ,\tilde{\alpha }}\bigr)={\mathbb{E}}
\bigl[X_{t}\bigm|{\mathscr{F}}^{ \alpha ,\tilde{\alpha }}_{T}\bigr]=
\frac{\alpha ^{2}}{\alpha ^{2} + \tilde{\alpha }^{2}}\int_{0}^{T} K(t,u)
\,dW^{\alpha ,\tilde{\alpha }}_{u}
\end{equation*}
and the deterministic covariance,
%
\begin{align}
\Upsilon (t,s) &=\int_{0}^{t\wedge s} \biggl( 1-
\frac{\alpha ^{2}}{\alpha ^{2} +
\tilde{\alpha }^{2}} \mbox{\large 1}_{[0,T]}(v) \biggr)^{2}
K(t,v)\, K(s,v) \,dv
\nonumber
\\
&\quad{}+ \frac{\alpha ^{2}\, \tilde{\alpha }^{2}}{(\alpha ^{2} +
\tilde{\alpha }^{2})^{2}} \int_{0}^{T}
K(t,v)\, K(s,v) \,dv.\label{eqn:cov-cond2}
\end{align}
\end{theorem}

\begin{remark}
Observe that the mean process $(\frac{\alpha ^{2}}{\alpha ^{2} +
\tilde{\alpha }^{2}}\int_{0}^{T} K(t,u) \,dW^{\alpha ,\tilde{\alpha
}}_{u})_{t\geq 0}$ is a continuous process. Therefore for almost every
continuous function $\psi $ defined on $[0,T]$,
%
\begin{equation}\label{eqn:mean-cond2}
m^{\psi }_{t}=\Psi _{t}(\psi )
\end{equation}
defines a continuous function $ {m}^{\psi }: [0,+\infty ) \longrightarrow
\mathbb{R}$. Thus, we can consider the continuous Gaussian process
$(X_{t}^{\psi })_{t\geq 0}$\index{continuous Gaussian process} with mean function ${m}^{\psi }$ and covariance
function~$\Upsilon $.

From continuity of ${m}^{\psi }$ one has
%
\begin{equation}\label{eqn:mean-limit-cond2}
\lim_{\varepsilon \to 0}{m}_{T+\varepsilon t}^{\psi }={m}_{T}^{\psi }
\end{equation}
uniformly for $t \in [0,1]$.
\end{remark}

\begin{remark}
Let us note that the covariance function of the conditioned process depends
on the time $T$, but not on the function ${\psi }$ as in the previous
section.
\end{remark}

\begin{remark}%
For $s\wedge t\geq T$ we have
%
\begin{equation}\label{eqn:cov-cond2-2}
\Upsilon (t,s) =\frac{\alpha ^{2}}{\alpha ^{2} + \tilde{\alpha }^{2}}
\int_{T}^{t\wedge s} K(t,v)\, K(s,v) \,dv + \frac{\tilde{\alpha }^{2}}{\alpha
^{2} + \tilde{\alpha }^{2}} \, k(t,s).
\end{equation}
For $\tilde{\alpha }=0$, i.e. ${\mathscr{F}}^{\alpha ,0}_{t}=\sigma \{X_{u}:
u\leq t\}$ (for details about the filtrations generated by $X$ and $B$, see,
for example, \cite{Yaz1}), we have the same conditioned variance as in
\cite{Gio-Pac}.
\end{remark}

\subsection{Large deviations}\label{sec4.2}

Let $\gamma _{\varepsilon }>0$ be an infinitesimal function, i.e. $\gamma
_{\varepsilon }\to 0$ for $\varepsilon \to 0$. In this section
$(X_{t})_{t\geq 0}$ is a continuous Volterra process\index{Volterra process} as in
(\ref{eqn:integral-representation}).

Now, in order to achieve a large deviation principle for the generalized
conditioned process $X^{\psi }$, we have to investigate the behavior of the
functions $\Upsilon $ and $m^{\psi }$ (defined in (\ref{eqn:cov-cond2}) and
(\ref{eqn:mean-cond2}), respectively) in a small time interval of length
$\varepsilon $. We want to investigate the behavior of the conditioned
process $(X_{t}^{\psi })_{t\geq 0}$ in the near future after~$T$.

For $s\leq t$, taking into account equation (\ref{eqn:cov-cond2-2}), simple
computations show that
%
\begin{align}
&{\operatorname{Cov}}\bigl( X^{\psi }_{T+\varepsilon t}-
X^{\psi }_{T}, X^{
\psi }_{T+\varepsilon s}-
X^{\psi }_{T}\bigr)
\nonumber
\\*
&\quad = \varepsilon \frac{\alpha ^{2}}{\alpha ^{2} + \tilde{\alpha }^{2}}\int_{0}^{t
\land s}
K(T+\varepsilon t,T+\varepsilon u) K(T + \varepsilon s, T+ \varepsilon u) \,du +\nonumber
\\*
&\qquad{} + \frac{\tilde{\alpha }^{2}}{\alpha ^{2} + \tilde{\alpha }^{2}} \bigl( {k(T+ \varepsilon t, T+\varepsilon s)
-k(T+\varepsilon t, T)-k(T, T+ \varepsilon s) +k(T, T)}\bigr).%
\label{eqn:covariance-conditional-process-epsilon}
\end{align}

Now we give conditions on the kernel $K$ of the original Volterra process\index{Volterra process} in
order to guarantee that the hypotheses of Theorem~\ref{th:ldp-gaussian} hold
for the conditioned process. The next assumption
(Assumption~\ref{ass:limit-kernel-cond2}) implies the existence of a limit
covariance.

\begin{assumption}\label{ass:limit-kernel-cond2}
For any $T>0$ there exists a square integrable function $\bar{K}$
(possibly~$0$) such that
\begin{equation*}
\lim_{\varepsilon \to 0}\sqrt{\varepsilon }\, \frac{K(T+\varepsilon t, T +
\varepsilon s)}{\gamma _{\varepsilon }\,} = \bar{K}(t,s)
\end{equation*}
uniformly in $(t,s)\in [0,1]\times [0,1]$.
\end{assumption}

\begin{remark}
Notice that we can choose $\gamma _{\varepsilon }$ so that $\lim_{\varepsilon
\to 0} \gamma _{\varepsilon }\varepsilon ^{- \frac{1}{2}}\in \{0,1,+\infty
\}$.
\end{remark}

The next assumption (Assumption \ref{ass:exp-tight-cond2}) implies the
exponential tightness\index{exponential tightness} of the family of processes.

\begin{assumption}\label{ass:exp-tight-cond2}
For any $T>0$ there exist constants $c,\hat{\tau }> 0$ such that
\begin{equation*}
\sup_{s,t \in [0,1],s\neq t } \frac{\int_{0}^{ t} (K(T+\varepsilon
t,T+\varepsilon u)-K(T+\varepsilon s,T+\varepsilon u))^{2} \, du}{\gamma
_{\varepsilon }^{2}|t-s|^{2\hat{\tau }}} \le c.
\end{equation*}
\end{assumption}

\begin{remark}\label{rem:condition-exp-tight-cond2}
Note that
\begin{align*}
&\int_{0}^{ t} \bigl(K(T+
\varepsilon t,T+\varepsilon u)-K(T+\varepsilon s,T+\varepsilon u)
\bigr)^{2} \, du
\\
&=\int_{s}^{t}\!\!
K(T\,{+}\,\varepsilon t,T\,{+}\,\varepsilon u)^{2} \, du\,{+}\, \int
_{0}^{ s} \bigl(K(T\,{+}\,\varepsilon t,T\,{+}\,\varepsilon
u)\,{-}\,K(T\,{+}\,\varepsilon s,T\,{+}\,\varepsilon u)\bigr)^{2} \, du,
\end{align*}
therefore in order to prove that Assumption \ref{ass:exp-tight-cond2} is
fulfilled we can prove that there exists $c>0$ such that
\begin{eqnarray*}
&&\sup_{s,t \in [0,1],s\neq t } \frac{\int_{ s}^{t} K(T+\varepsilon
t,T+\varepsilon u)^{2} \, du }{\gamma _{\varepsilon
}^{2}|t-s|^{2\hat{\tau }}} \le c,
\\
&&\sup_{s,t \in [0,1],s\neq t } \frac{\int_{0}^{ s} (K(T+\varepsilon
t,T+\varepsilon u)-K(T+\varepsilon s,T+\varepsilon u))^{2} \, du}{\gamma
_{\varepsilon }^{2}|t-s|^{2\hat{\tau }}} \le c.
\end{eqnarray*}
\end{remark}

\begin{proposition}\label{prop:limit-cov-cond2}
Under Assumptions \ref{ass:limit-cov} and \ref{ass:limit-kernel-cond2} one
has
\begin{equation*}
\lim_{\varepsilon \to 0} \frac{\operatorname{Cov}(X^{\psi }_{T+\varepsilon
t}-X^{\psi }_{T}, X^{\psi }_{T+\varepsilon s}-X^{\psi }_{T})}{\gamma
_{\varepsilon }^{2}}= \bar{\Upsilon }(t,s)
\end{equation*}%
uniformly in $(t,s)\in [0,1]\times [0,1]$, where
%
\begin{equation}\label{eqn:cov-limit-cond2}
\bar{\Upsilon }(t,s)= \frac{\tilde{\alpha }^{2}}{\alpha ^{2} + \tilde{\alpha
}^{2}}\int _{0}^{t \land s} \bar{K}(t,u) \bar{K}(s,u) \, du+
\frac{\tilde{\alpha }^{2}}{\alpha ^{2} + \tilde{\alpha }^{2}}\,\bar{k}(t,s).
\end{equation}
\end{proposition}

\begin{proof}
From (\ref{eqn:covariance-conditional-process-epsilon}), Assumption
\ref{ass:limit-cov} and Assumption \ref{ass:limit-kernel-cond2} the claim
easily follows.
\end{proof}

\begin{proposition}\label{prop:exp-tight-cond2}
Under Assumptions \ref{ass:exp-tight} and \ref{ass:exp-tight-cond2} the
family $((X^{\psi }_{T+\varepsilon t}-X^{\psi }_{T} - {\mathbb{E}}[X^{\psi
}_{T+ \varepsilon t}-X^{\psi }_{T}])_{t\in [0,1]})_{\varepsilon >0}$ is
exponentially tight at 
the inverse speed $\gamma _{\varepsilon }^{2}$.
\end{proposition}

\begin{proof}
Since $((X^{\psi }_{T+\varepsilon t}-X^{\psi }_{T} - {\mathbb{E}}[X^{\psi
}_{T+ \varepsilon t}-X^{\psi }_{T}])_{t\in [0,1]})_{\varepsilon >0}$ is a
family of centered processes, it is enough to prove that
(\ref{eqn:cov-condition-exp-tight}) is satisfied with 
an appropriate
speed function.

The covariance of such process is given by
(\ref{eqn:covariance-conditional-process-epsilon}). Therefore
\begin{align*}
& {\operatorname{Var}}\bigl(X^{\psi }_{T+\varepsilon t}-X^{\psi }_{T+
\varepsilon s}
\bigr)
\\
&\quad = \varepsilon \frac{\alpha ^{2}}{\alpha ^{2} + \tilde{\alpha }^{2}} \int_{0}^{t}
\bigl(K(T+ \varepsilon t,T+\varepsilon u) - K(T+\varepsilon s,T+\varepsilon u)
\bigr)^{2} \,du +
\\
&\qquad{} + \frac{\tilde{\alpha }^{2}}{\alpha ^{2} + \tilde{\alpha }^{2}}k(T+ \varepsilon t, T+\varepsilon t) -2 k(T+
\varepsilon t, T+\varepsilon s ) +k(T+\varepsilon s , T+\varepsilon s).%
\end{align*}
From Assumption \ref{ass:exp-tight} we already know that
\begin{equation*}
\sup_{s,t\in [0,1],s\neq t} \frac{|k(T+\varepsilon t, T+\varepsilon t) -2
k(T+\varepsilon t, T+\varepsilon s ) +k(T+\varepsilon s , T+\varepsilon
s)|}{\gamma _{\varepsilon }^{2} \,|t-s|^{2\tau }} \leq M.
\end{equation*}
Furthermore, Assumption \ref{ass:exp-tight-cond2} implies that
\begin{equation*}
\sup_{s,t\in [0,1],s\neq t} \frac{\int_{0}^{t} (K(T+\varepsilon
t,T+\varepsilon u) - K(T+\varepsilon s,T+\varepsilon u))^{2} \,du}{\gamma
_{\varepsilon }^{2} \,|t-s|^{2\hat{\tau }}} \leq M.
\end{equation*}

Therefore condition (\ref{eqn:cov-condition-exp-tight}) holds with the
inverse speed $\eta _{\varepsilon }=\gamma _{\varepsilon }^{2}$ and $\beta
=\tau \wedge \hat{\tau }$.
\end{proof}

We are ready to state a large deviation principle for the conditioned Vol\-ter\-ra
process.

\begin{theorem}\label{th:ldp-cond2}
Suppose Assumptions \ref{ass:limit-cov}, \ref{ass:exp-tight},
\ref{ass:limit-kernel-cond2} and \ref{ass:exp-tight-cond2} are fulfilled. If
the (existing) covariance function $\bar{\Upsilon }$ defined in
Proposition~\ref{prop:limit-cov-cond2} is regular enough, then the family of
processes $((X^{\psi }_{T+\varepsilon t}-X^{\psi }_{T})_{t\in
[0,1]})_{\varepsilon >0}$ satisfies a large deviation principle on $C[0,1]$
with the inverse speed $\gamma _{\varepsilon }^{2}$ and the good rate
function
%
\begin{equation}\label{eqn:rate-cond2}
I(h) = %
\begin{cases}
\frac{1}{2}  \llVert   h \rrVert   _{\bar{{\mathscr{H}}}}^{2}, & h \in
\bar{{\mathscr{H}}},
\\
+\infty, & \text{otherwise},
\end{cases} %
\end{equation}
where $\bar{{\mathscr{H}}}$ and $ \llVert   .  \rrVert
_{\bar{{\mathscr{H}}}}$, respectively, denote the reproducing kernel\index{reproducing kernel Hilbert space} Hilbert
space and the related norm associated to the covariance function
$\bar{\Upsilon }$ given by (\ref{eqn:cov-limit-cond2}).
\end{theorem}

\begin{proof}
Cosider the family of centered processes $((X^{\psi }_{T+\varepsilon
t}-X^{\psi }_{T} - {\mathbb{E}}[X^{\psi }_{T+ \varepsilon t}-\break X^{\psi
}_{T}])_{t\in [0,1]})_{\varepsilon }$. Thanks to
Proposition~\ref{prop:exp-tight-cond2} this family of processes is
exponentially tight at the inverse speed $\gamma _{\varepsilon }^{2}$. Thanks
to Proposition~\ref{prop:limit-cov-cond2}, for any $\lambda \in
{\mathscr{M}}[0,1]$, one has
\begin{equation*}
\lim_{\varepsilon \to 0} \frac{\operatorname{Var}( \langle \lambda , X^{\psi
}_{T+\varepsilon t}-X^{\psi }_{T} \rangle )}{\gamma _{\varepsilon }^{2}} =
\int _{0}^{1} \int_{0}^{1} \bar{\Upsilon }(t,s) \, d\lambda (t) \,d \lambda
(s)
\end{equation*}
where $\bar{\Upsilon }$ is defined in (\ref{eqn:cov-limit-cond2}) Since
$\bar{\Upsilon }$ is the covariance function of a continuous\break Volterra
process,\index{Volterra process} a large deviation principle for $((X^{\psi }_{T+\varepsilon
t}-X^{\psi }_{T}- {\mathbb{E}}[X^{\psi }_{T+ \varepsilon t}-\break X^{\psi
}_{T}])_{t\in [0,1]})_{t\in [0,1]})_{ \varepsilon >0}$ actually holds from
Theorem~\ref{th:ldp-gaussian} with the inverse speed $\gamma _{\varepsilon
}^{2}$ and the good rate function given by (\ref{eqn:rate-cond2}). From
Equation~(\ref{eqn:mean-limit-cond2}) and Remark~\ref{rem:exp-equiv} the same
large deviation principle holds for the noncentered family $((X^{\psi
}_{T+\varepsilon t}-X^{\psi }_{T})_{t\in [0,1]})_{\varepsilon >0}$.
\end{proof}

\subsection{Examples}\label{sec4.3}

In this section we consider some examples to which Theorem \ref{th:ldp-cond2}
applies. Therefore we want to verify that Assumptions \ref{ass:limit-cov},
\ref{ass:exp-tight}, \ref{ass:limit-kernel-cond2} and
\ref{ass:exp-tight-cond2} are fulfilled. Let $X$ be a continuous, centered
Volterra process process with kernel $K$.

\begin{example}[Fractional Brownian Motion]\label{ex:FBM-cond2}
Consider a fractional Brownian motion\index{fractional Brownian motion} with $H>1/2$ as in
Example~\ref{ex:FBMcond1}. We have already proved that
Assumptions~\ref{ass:limit-cov} and~\ref{ass:exp-tight} are fulfilled with
$\tau =H$ and $\gamma _{\varepsilon }=\varepsilon ^{H}$.

We want to show that Assumptions~\ref{ass:limit-kernel-cond2} and
\ref{ass:exp-tight-cond2} are fulfilled with $\hat{\tau }=1$, $\gamma
_{\varepsilon }=\varepsilon ^{H}$. From Example 4.17 in \cite{Gio-Pac}, we
have that
\begin{equation*}
\lim_{\varepsilon \to 0} \sqrt{\varepsilon }\, \frac{K(T+\varepsilon
t,T+\varepsilon s)}{\varepsilon ^{H}} = c_{H} (t-s)^{H- \frac{1}{2}}
\end{equation*}
uniformly for $t,s \in [0,1]$. Therefore
\begin{align*}
&\lim_{\varepsilon \to 0}\frac{1}{\varepsilon ^{2H}}\int _{0}^{t \land s}
K(T+\varepsilon t,T+\varepsilon u) K(T + \varepsilon s, T+\varepsilon u) \,du
\\
&\quad= c_{H}^{2} \int_{0}^{t \land s} (t-u)^{H-\frac{1}{2}}
(s-u)^{H-\frac{1}{2}} \, du
\end{align*}
uniformly for $(t,s)\in [0,1]\times [0,1]$. So, we have that Assumption
\ref{ass:limit-kernel-cond2} is fulfilled with
\begin{equation*}
\bar{\Upsilon }(t,s)= \frac{\tilde{\alpha }^{2}}{\alpha ^{2} + \tilde{\alpha
}^{2}}k(t,s) + \frac{\alpha ^{2}}{\alpha ^{2} + \tilde{\alpha }^{2}}
c_{H}^{2} \int_{0}^{t \land s} (t-u)^{H-\frac{1}{2}} (s-u)^{H-\frac{1}{2}} \,
du.
\end{equation*}
Note that $\bar{\Upsilon }$ is regular enough. Let us now prove the
exponential tightness.\index{exponential tightness}

Since $(a+b)^{2} \le 2(a^{2}+b^{2})$ for $ a,b \in \mathbb{R}$, from
Equation~(\ref{eqn:kernel fbm}), there exists a constant $c>0$ such that, for
$s<t$,
\begin{eqnarray*}
&&K(T+\varepsilon t,T+\varepsilon u)^{2}
\\
&\le & c \Biggl( \varepsilon ^{2H-1}( t - u)^{2H-1} +
\Biggl( \int_{T+
\varepsilon u}^{T+\varepsilon t} \bigl(v-(T+\varepsilon u)
\bigr)^{H-\frac{1}{2}} \, dv \Biggr)^{2} \Biggr)
\\
&\le & c \bigl( \varepsilon ^{2H-1}( t - u)^{2H-1} +
\varepsilon ^{2H+1}( t - u)^{2H+1} \bigr).
\end{eqnarray*}
Thus,
\begin{equation*}
\varepsilon \int_{s}^{t} K(T+\varepsilon t,T+ \varepsilon u)^{2} \, du \le c
\bigl( \varepsilon ^{2H} (t-s)^{2H} + \varepsilon ^{2H+2} (t-s)^{2H+2}
\bigr).
\end{equation*}
Furthermore,
\begin{equation*}
\bigl(K(T+\varepsilon t,T+\varepsilon u)-K(T+\varepsilon s,T+\varepsilon u)
\bigr)^{2}= \bigl(A(u)- B(u)\bigr)^{2}\leq 2\bigl( A^{2}(u)+ B^{2}(u)\bigr),
\end{equation*}
where
\begin{align*}
&A(u) = c_{H} \biggl[ \biggl( \frac{T+\varepsilon t}{T +\varepsilon u}
\biggr)^{H-\frac{1}{2}}\!\!\! ( t - u)^{H-\frac{1}{2}}\varepsilon ^{H-
\frac{1}{2}}
- \biggl( \frac{T+\varepsilon s}{T +\varepsilon u} \biggr)^{H-
\frac{1}{2}}\!\!\! (
s-u)^{H-\frac{1}{2}}\varepsilon ^{H-\frac{1}{2}} ) \biggr]
\\
&B(u) = c_{H} \biggl( H-\frac{1}{2} \biggr)
\frac{1}{(T+\varepsilon u)^{H-\frac{1}{2}}} \int_{T+\varepsilon s}^{T+
\varepsilon t}
v^{H-\frac{3}{2}} \bigl(v-(T+\varepsilon u)\bigr)^{H-\frac{1}{2}} \, dv.
\end{align*}
Now, thanks to the Lagrange theorem,\index{Lagrange theorem} there exists $x\in [s,t]$ such that
\begin{align*}
A(u) &\leq c\,\bigl( (T+\varepsilon t)^{H-\frac{1}{2}} ( t-u
)^{H-\frac{1}{2}}\varepsilon ^{H-\frac{1}{2}} - (T+\varepsilon
s)^{H-\frac{1}{2}} ( s-u)^{H-\frac{1}{2}}\varepsilon ^{H-\frac{1}{2}}\bigr)
\\
&= c \bigl((T+\varepsilon x)^{H-\frac{3}{2}}(x- u)^{H-\frac{1}{2} } +(T+
\varepsilon x)^{H-\frac{1}{2}} (x- u)^{H-\frac{3}{2}}\bigr) \varepsilon
^{{H-\frac{1}{2}}}(t-s).
\end{align*}

The estimation of $B$ is easily done. There exists $x\in [s,t]$ such that
\begin{align*}
B(u) \le c \int_{T+\varepsilon s}^{T+\varepsilon t}
\bigl(v-(T+\varepsilon u)\bigr)^{H-\frac{1}{2}} \, dv&=c \varepsilon
^{H+\frac{1}{2}} \int_{ s}^{ t} (v-
u)^{H-\frac{1}{2}}\, dv
\\
&= c \varepsilon ^{H+\frac{1}{2}}(x-
u)^{H-\frac{3}{2}}(t-s),
\end{align*}
and then
\begin{equation*}
\varepsilon \int_{0}^{ s} A^{2}(u) \, du \le c \varepsilon ^{2H} (t-s)^{2},
\qquad \varepsilon \int_{0}^{ s} B^{2}(u) \, du\le c \varepsilon ^{2H+2}
(t-s)^{2}.
\end{equation*}
From Remark \ref{rem:condition-exp-tight-cond2} we have that
\begin{equation*}
\int_{0}^{ t} \bigl(K(T+ \varepsilon t,T+\varepsilon u)-K(T+\varepsilon
s,T+\varepsilon u) \bigr)^{2} \, du\leq c \varepsilon ^{2H} (t-s)^{2}.
\end{equation*}
Assumption \ref{ass:exp-tight-cond2} is then fulfilled with $\hat{\tau }=1$
and $\gamma _{\varepsilon }=\varepsilon ^{H}$.

A large deviation principle is then established for the family of processes\break
$((X^{\psi }_{T+\varepsilon t}-X^{\psi }_{T})_{t\in [0,1]})_{\varepsilon
>0}$, with the inverse speed $\varepsilon ^{2H}$.
\end{example}

\begin{example}[$m$-fold integrated Brownian motion]\label{ex:mIBM-cond2}
Let $X$ be the process defined in Example \ref{ex:mIBM-cond1}. We have
already proved that Assumptions~\ref{ass:limit-cov} and~\ref{ass:exp-tight}
are fulfilled with $\tau =1$ and $\gamma _{\varepsilon }=\varepsilon $. We
want to show that Assumptions~\ref{ass:limit-kernel-cond2}
and~\ref{ass:exp-tight-cond2} are fulfilled with $\hat{\tau }=1$,
$\gamma_{\varepsilon }=\varepsilon $. For $s\leq t$, \xch{we}{e} have
\begin{equation*}
{K(T+\varepsilon t, T + \varepsilon s)} =(t-s)^{m}{\varepsilon ^{m}},
\end{equation*}
therefore
\begin{equation*}
\lim_{\varepsilon \to 0} \sqrt{\varepsilon }\, \frac{K(T+\varepsilon
t,T+\varepsilon s)}{{\varepsilon }} = 0
\end{equation*}
uniformly in $t,s \in [0,1]$. Thus,
\begin{equation*}
\bar{\Upsilon }(t,s)= \frac{\tilde{\alpha }^{2}}{\alpha ^{2} + \tilde{\alpha
}^{2}} \frac{1}{(m!)^{2}} \frac{m^{2}}{2m-1}T^{2m-1} st,
\end{equation*}
and Assumption \ref{ass:limit-kernel-cond2} is verified. Note that
$\bar{\Upsilon }$ is regular enough. Let us now prove the exponential
tightness\index{exponential tightness} of the family of processes. For $s< t$, there exist positive
constants $c_{1},c_{2}$ such that
\begin{align*}
&\int_{0}^{ t} \bigl(K(T+
\varepsilon t,T+\varepsilon u)-K(T+\varepsilon s,T+\varepsilon u)
\bigr)^{2} \, du\\
&\quad= \frac{1}{(m!)^{2}}\int_{0}^{s} \varepsilon ^{2m} \bigl((t-u)^{m}-
(s-u)^{m} \bigr)^{2} \,du \\
&\qquad{} + \frac{1}{(m!)^{2}}\int_{s}^{t} \varepsilon
^{2m} ((t-u)^{2m} \,du \leq \varepsilon ^{2m} c_{1} (t-s)^{2} + \varepsilon
^{2m}c_{2} (t-s)^{2m+1}
\end{align*}
and Assumption \ref{ass:exp-tight-cond2} is verified with $\hat{\tau }=1$ ad
$\gamma _{\varepsilon }=\varepsilon $. A large deviation principle is then
established for the family of conditioned processes $((X^{\psi
}_{T+\varepsilon t}-X^{\psi }_{T})_{t\in [0,1]})_{\varepsilon >0}$, with the
inverse speed $\varepsilon ^{2}$.
\end{example}

\begin{example}[Integrated Volterra Process]\label{ex:IVP-cond2}
Let $X$ be the process defined in Example~\ref{ex:IVP-cond1}. We have already
proved that Assumptions \ref{ass:limit-cov} and \ref{ass:exp-tight} are
fulfilled with $\tau =1$, $\gamma _{\varepsilon }=\varepsilon $. We want to
show that Assumptions \ref{ass:limit-kernel-cond2} and
\ref{ass:exp-tight-cond2} are fulfilled with $\hat{\tau }=1$ and $\gamma
_{\varepsilon }=\varepsilon $. For $s\leq t$, \xch{we}{e} have
\begin{equation*}
{h(T+\varepsilon t, T + \varepsilon s)} =\int_{T+\varepsilon s}^{T+
\varepsilon t} K(v, T+\varepsilon s) \,dv=\varepsilon \int_{s}^{ t} K(T+
\varepsilon v, T+\varepsilon s) \,dv,
\end{equation*}
therefore
\begin{equation*}
\lim_{\varepsilon \to 0} \sqrt{\varepsilon }\, \frac{h(T+\varepsilon t, T +
\varepsilon s)}{{\varepsilon }} = 0
\end{equation*}
uniformly in $t,s \in [0,1]$. Thus,
\begin{equation*}
\bar{\Upsilon }(t,s)= \frac{\tilde{\alpha }^{2}}{\alpha ^{2} + \tilde{\alpha
}^{2}}\int_{0}^{T} K^{2}(T,u) \,du\, st,
\end{equation*}
and Assumption \ref{ass:limit-kernel-cond2} is verified. Note that
$\bar{\Upsilon }$ is regular enough. Let us now prove the exponential
tightness\index{exponential tightness} of the family of processes. For $s< t$, there exist a constant
$c>0$ such that
\begin{align*}
&\int_{0}^{ t} \bigl(h(T+
\varepsilon t,T+\varepsilon u)-h(T+\varepsilon s,T+\varepsilon u)
\bigr)^{2} \, du
\\
&\quad=\int_{0}^{ s}
\Biggl(\int_{T+\varepsilon s}^{T+\varepsilon t} K(v, T+\varepsilon u) \,dv
\Biggr)^{2}\,du+\int_{s}^{t} \Biggl(
\int_{T+\varepsilon u}^{T+\varepsilon t} K(v, T+\varepsilon u) \,dv
\Biggr)^{2}\,du
\\*
&\quad\leq \varepsilon ^{2}c
(t-s)^{2}
\end{align*}
and Assumption \ref{ass:exp-tight-cond2} is verified with $\hat{\tau }=1$ and
$\gamma _{\varepsilon }=\varepsilon $. A large deviation principle is then
established for the family of conditioned processes $((X^{\psi
}_{T+\varepsilon t}-X^{\psi }_{T})_{t\in [0,1]})_{\varepsilon >0}$, with the
inverse speed $\varepsilon ^{2}$.
\end{example}



\begin{acknowledgement}[title={Acknowledgments}]
The author wish to thank the Referees for their very useful comments which
allowed me to greatly improve the presentation of the paper.
\end{acknowledgement}

\begin{funding}
The author acknowledges the \gsponsor[id=GS1,sponsor-id=501100003407]{MIUR} Excellence Department Project awarded to the
Department of Mathematics, University of Rome Tor Vergata, CUP
\gnumber[refid=GS1]{E83C180001\allowbreak00006}.
\end{funding}


\end{document}